\documentclass[11pt]{article}

\usepackage[utf8]{inputenc}
\usepackage[english]{babel}
\usepackage{hyperref}

\usepackage{amssymb,amsmath,amsthm}
\usepackage{enumerate}
\usepackage{enumitem}
\usepackage{algorithm}
\usepackage{algorithmic}
\usepackage{array}
\usepackage{graphicx, color}
\usepackage{epsfig}
\usepackage{tikz}
\usetikzlibrary{calc,decorations.pathmorphing,decorations.text}
\usepackage{subcaption}
\usepackage{refcount} 
\usepackage[hmargin=2.5cm,vmargin=3cm]{geometry}
\usepackage{mathtools, braket}
\usepackage[normalem]{ulem}
\usepackage{authblk}
\usepackage{centernot}
\usepackage{indentfirst}
\usepackage{framed} 
\usepackage{amsmath}
\allowdisplaybreaks[4]
\usepackage[nameinlink]{cleveref}

\usepackage{hyperref}
\hypersetup{
    hidelinks,
	colorlinks=true, 
	linkcolor=blue,
	citecolor=red}

\theoremstyle{plain}
\newtheorem{thm}{Thm}[section]

\newtheorem{theorem}[thm]{Theorem}
\newtheorem{lemma}[thm]{Lemma}
\newtheorem{corollary}[thm]{Corollary}
\newtheorem{proposition}[thm]{Proposition}

\newtheorem{conjecture}[thm]{Conjecture}
\newtheorem{problem}[thm]{Problem}

\newtheorem{observation}[thm]{Observation}
\newtheorem{definition}[thm]{Definition}

\usepackage{bm}

\usepackage[nameinlink]{cleveref}

\newtheorem{innercustomgeneric}{\customgenericname}
\providecommand{\customgenericname}{}
\newcommand{\newcustomtheorem}[2]{%
	\newenvironment{#1}[1]
	{%
		\renewcommand\customgenericname{#2}%
		\renewcommand\theinnercustomgeneric{##1}%
		\innercustomgeneric
	}
	{\endinnercustomgeneric}
}

\newcustomtheorem{customtheorem}{Theorem}
\newcustomtheorem{customlemma}{Lemma}

\newenvironment{proof*}{\noindent \emph{Proof.}}{\hfill$\Diamond$}

\renewcommand{\pod}[1]{\allowbreak\mathchoice
	{\if@display \mkern 0mu\else \mkern 0mu\fi (#1)}
	{\if@display \mkern 0mu\else \mkern 0mu\fi (#1)}
	{\mkern 1mu(\mathrm{mod}\mkern 4mu #1)}
	{\mkern 0mu(#1)}
}

\usepackage[color=green!30]{todonotes}

\usepackage{caption}
\usetikzlibrary{matrix}
\usetikzlibrary{decorations.markings}
\tikzstyle{vertex}=[circle, draw, fill=black!50,
inner sep=0pt, minimum width=4pt]
\tikzset{->-/.style={decoration={
			markings,
			mark=at position .5 with {\arrow{>}}},postaction={decorate}}}

\tikzstyle{bigblue}=[color=blue, very thick, >=stealth]
\tikzstyle{lightblue}=[color=blue, thin, >=stealth]

\tikzstyle{bigred}=[color=red, very thick, >=stealth]
\tikzstyle{lightred}=[color=red, thin, >=stealth]

\tikzstyle{biggreen}=[color=black!30!green, very thick, >=stealth]
\tikzstyle{lightgreen}=[color=black!30!green,  thin, >=stealth]

\def\Z{\mathbb{Z}}

\def\R{\mathbb{R}}

\usepackage{amsmath, amssymb, array, booktabs, caption}

\newcolumntype{M}[1]{>{\centering\arraybackslash}m{#1}}

\title{Reduction Operations and Characterizations of $S^1$-Flows in Graphs}
\author{Chenxing Li$^1$, Jiaao Li$^1$, Rong Luo$^2$, and Bo Su$^1$\\
\small $^1$School of Mathematical Sciences and LPMC, Nankai University, Tianjin 300071, China \\
\small $^2$Department of Mathematics, West Virginia University, Morgantown 26505, USA\\
\small Emails: chenxingli@mail.nankai.edu.cn; lijiaao@nankai.edu.cn; rluo@mail.wvu.edu; suboll@mail.nankai.edu.cn
}
\date{}

\begin{document}
\maketitle

\mbox{}
\vspace{-1.5cm}
\begin{abstract}
Thomassen (J. Combin. Theory Ser. B 108 (2014), 81–91) showed that  every graph admitting a nowhere-zero $3$-flow also admits an
$S^1$-flow. He also proved that the converse holds for cubic graphs, but constructed counterexamples showing it fails in general.  Wang et al. (SIAM J. Discrete Math. 29 (2015), 2166–2178) presented a couple of sufficient conditions under which the existence of an $S^1$-flow guarantees the existence of a nowhere-zero 3-flow. In this paper,  we first prove that a graph with maximum degree at most four admits a nowhere-zero $3$-flow if and only if it admits an $S^1$-flow.
We
then develop some reduction techniques for $S^1$-flows based on graph operations
including bull-growth, $2$-sums, and  contractions.  Finally, we
apply those techniques to characterize  triangularly connected graphs and graphs
containing a spanning triangle-tree that admit $S^1$-flows, respectively.

\end{abstract}

\section{Introduction}

All graphs considered in this paper are finite and loopless. Parallel edges are allowed unless stated otherwise. Contractions
delete the resulting loops but retain parallel edges.

Tutte~\cite{Tutte1949,Tutte54} introduced the theory of integer flows as a
dual counterpart to vertex colorings. This theory was later extended to
real-valued flows by Goddyn, Tarsi, and Zhang~\cite{Goddyn}, and subsequently
generalized to vector flows by Jain~\cite{Jain2007} and
Thomassen~\cite{thomassen2014group}.

For a graph $G$, let $V(G)$ and $E(G)$ denote its vertex set and edge set, respectively. We write $N_G(v)$ for the set of neighbors of $v$ in $G$, and $d_G(v)$ for the degree of $v$ in $G$. A vertex of degree $k$ (or at least $k$, or at most $k$) is called a $k$-vertex (or a $k^+$-vertex, or a $k^-$-vertex, respectively). Let $V_k(G)$ denote the set of $k$-vertices in $G$. Given an orientation $D$ of $G$, let $E_D^+(v)$ and $E_D^-(v)$ denote the sets of edges directed out of and into a vertex $v$, respectively.

\begin{definition}
Let $\Omega\subseteq\mathbb{R}^d$, and let $G$ be a graph with an
orientation $D$. For a vector assignment $f:E(G)\to\mathbb{R}^d$, define
the \emph{boundary} of $f$ at a vertex $v$ by
\[
\partial_D f(v):=
\sum_{e\in E_D^+(v)}f(e)
-
\sum_{e\in E_D^-(v)}f(e).
\]
An ordered pair $(D,f)$ is called a \emph{vector $\Omega$-flow} of $G$
if $f:E(G)\to\Omega$ and
\[
\partial_D f(v)=\bm{0}
\]
for every vertex $v\in V(G)$.
\end{definition}

When the context is clear, we may omit the subscript in some notation and refer to a vector $\Omega$-flow simply as an \emph{$\Omega$-flow}. If $\Omega=\{\pm1,\dots,\pm(k-1)\}\subseteq\mathbb{R}$ for some integer $k$, then an $\Omega$-flow is the classical \emph{nowhere-zero $k$-flow}.

Vector flows have attracted considerable attention in recent years. Topics
that have been studied include geometric characterizations, $p$-normed
generalizations, and progress on Jain's conjectures;
see~\cite{GMRR2025,HMM2026,LLLS2026,mattiolo2023d,mattiolo2024lower,
mattiolo2025geometric,Ulyanov2026,wang2015vector} and the references therein.

This paper focuses on $S^1$-flows, where
\[
S^1=\{\boldsymbol{x}\in\mathbb{R}^2:\|\boldsymbol{x}\|=1\}
\]
is the unit circle in $\mathbb{R}^2$, and $\|\boldsymbol{x}\|$ denotes the
Euclidean norm of $\boldsymbol{x}$.

Thomassen~\cite{thomassen2014group} established the following fundamental
relationship between nowhere-zero $3$-flows and $S^1$-flows.

\begin{theorem}[\cite{thomassen2014group}]
\label{thm:3-R3-S1-flow}
Let $G$ be a graph, and let $R_3$ denote the set of cube roots of unity. Then
~\textup{(a)} and~\textup{(b)} are equivalent, and each
implies statement~\textup{(c)}.
\begin{enumerate}[label=(\alph*)]
    \item $G$ admits a nowhere-zero $3$-flow;
    \item $G$ admits an $R_3$-flow;
    \item $G$ admits an $S^1$-flow.
\end{enumerate}
Moreover, if $G$ is cubic, then the three statements are equivalent.
\end{theorem}

Note that Theorem~\ref{thm:3-R3-S1-flow} immediately gives that a graph $G$ with maximum degree at most $3$  admits an $S^1$-flow if and only if $G$ admits a nowhere-zero $3$-flow.
Thomassen also constructed a family of graphs that admit $S^1$-flows but no
nowhere-zero $3$-flows. This naturally leads to the problem of determining
when the existence of an $S^1$-flow guarantees the existence of a nowhere-zero
$3$-flow. Motivated by this gap, Wang et al.~\cite{wang2015vector} posed the
following problem.

\begin{problem}[\cite{wang2015vector}]
Characterize conditions under which the existence of an $S^1$-flow guarantees
the existence of a nowhere-zero $3$-flow. The conditions may be considered
from the following two perspectives:
\begin{enumerate}[label=(\arabic*)]
    \item properties of the $S^1$-flow itself;
    \item structural properties of the underlying graph.
\end{enumerate}
\end{problem}

Wang et al.~\cite{wang2015vector} obtained the following sufficient
conditions.

\begin{theorem}[\cite{wang2015vector}]
\label{thm:wang}
Let $G$ be a graph.
\begin{enumerate}[label=(\arabic*)]
    \item \label{thm:wang-vector-S1}
    If $G$ admits an $S^1$-flow of rank at most two, then it admits a
    nowhere-zero $3$-flow.

    \item \label{thm:wang-degree3}
    If $G[V_3(G)]$ is connected and $G-V_3(G)$ is acyclic, then $G$ admits an
    $S^1$-flow if and only if it admits a nowhere-zero $3$-flow.
\end{enumerate}
\end{theorem}

The rank of an $S^1$-flow appearing in Theorem~\ref{thm:wang} is defined
in~\cite{wang2015vector}. Since it is not used elsewhere in this paper, we
omit the definition.

Motivated by these results, we investigate the existence of $S^1$-flows in
graphs with particular structural properties. 
 Our first result establishes
the equivalence between $S^1$-flows and nowhere-zero $3$-flows for all graphs with 
 maximum degree at most~$4$.

\begin{theorem}
\label{thm:degree4}
Let $G$ be a graph with maximum degree at most $4$. Then $G$ admits a
nowhere-zero $3$-flow if and only if it admits an $S^1$-flow.
\end{theorem}

We then develop several reduction techniques for $S^1$-flows; see
Section~\ref{sec:reduction}. These techniques are of independent interest and
may be useful in future work. To the best of our knowledge, they are the first
reduction techniques developed specifically for the study of $S^1$-flows. We
apply them to characterize triangularly connected graphs and  graphs
containing a spanning triangle-tree that admit an $S^1$-flow, respectively.

We first introduce the terminology needed to state these characterizations. In the following definition, the graphs under consideration are simple.

\begin{definition}\label{path}
\begin{enumerate}[label=(\arabic*)]
    \item \label{def:triangle-tree} A \emph{triangle-tree} is a graph obtained from a triangle by
    repeatedly adding a new vertex adjacent to the two endpoints of an
    existing edge. A $2$-vertex in a triangle-tree is called a \emph{leaf}.

    \item \label{def:triangle-path} A \emph{triangle-path} is either a triangle or a triangle-tree with
    exactly two leaves. A triangle-path with leaves $u$ and $v$ is denoted by
    $\mathcal{P}(u,v)$.

    \item \label{def:triangularly-connected} A connected graph $G$ is \emph{triangularly connected} if every pair of edges of $G$, not necessarily distinct, lies in a common triangle-path.

    \item \label{def:crystal} A \emph{crystal} is a graph obtained from a triangle-path $\mathcal{P}(u,v)$  with at least four vertices by adding a new edge connecting its two leaves $u,v$, denoted   
    $\mathcal{C}=\mathcal{P}(u,v)+uv$. A crystal $\mathcal{C}$ is called
    an \emph{odd crystal} if every vertex of $\mathcal{C}$ has odd degree; see
    Figure~\ref{fig:crystal}.
\end{enumerate}
\end{definition}

\begin{figure}[htbp]
\centering

\begin{subfigure}{0.38\textwidth}
\centering
\begin{tikzpicture}[>=latex,
    roundnode/.style={circle, draw=black!90, thick, minimum size=1.5mm, inner sep=0pt},
    scale=0.45
]

    \node [roundnode] (1) at (-3,0) {};
    \node [roundnode] (2) at (-2,1.732) {};
    \node [roundnode] (3) at (0,1.732) {};
    \node [roundnode] (4) at (-1,0) {};
    \node [roundnode] (5) at (1,0) {};
    \node [roundnode] (6) at (0,-1.732) {};
    \node [roundnode] (7) at (2,-1.732) {};
    \node [roundnode] (8) at (3,0) {};

    \draw[line width=0.8pt] (4)--(1)--(2)--(3)--(5)--(8);
    \draw[line width=0.8pt] (2)--(4)--(3);
    \draw[line width=0.8pt] (4)--(5)--(6)--(4);
    \draw[line width=0.8pt] (5)--(7)--(8);
    \draw[line width=0.8pt] (6)--(7);

    \draw[line width=0.8pt]
        (1) .. controls +(-1,4) and +(1,4) .. (8);

\end{tikzpicture}
\caption{An odd crystal}
\end{subfigure}
\hfill
\begin{subfigure}{0.55\textwidth}
\centering
\begin{tikzpicture}[>=latex,
    roundnode/.style={circle, draw=black!90, thick, minimum size=1.5mm, inner sep=0pt},
    roundnode0/.style={circle, draw=black!90, fill=black, thick, minimum size=1.5mm, inner sep=0pt},
    scale=0.45
]

    \node[roundnode] (1) at (-6,0) {};
    \node[roundnode] (2) at (-5,1.732) {};
    \node[roundnode] (3) at (-3,1.732) {};
    \node[roundnode] (4) at (-4,0) {};
    \node[roundnode] (5) at (-2,0) {};
    \node[roundnode] (6) at (-3,-1.732) {};
    \node[roundnode0] (7) at (-1,-1.732) {};
    \node[roundnode] (8) at (0,0) {};

    \node[roundnode] (9) at (1,1.732) {};
    \node[roundnode] (10) at (3,1.732) {};
    \node[roundnode] (11) at (2,0) {};
    \node[roundnode] (12) at (4,0) {};

    \node[roundnode] (13) at (1,-1.732) {};

    \draw[line width=0.8pt]
        (4)--(1)--(2)--(3);
    \draw[line width=0.8pt]
        (2)--(4)--(3);
    \draw[line width=0.8pt]
        (4)--(5)--(3);
    \draw[line width=0.8pt]
        (4)--(5)--(6)--(4);
    \draw[line width=0.8pt]
        (5)--(7)--(6);
    \draw[line width=0.8pt]
        (7)--(8)--(5);
    \draw[line width=0.8pt]
        (7)--(13);
    \draw[line width=0.8pt]
        (11)--(8)--(9)--(10)--(11)--(9);
    \draw[line width=0.8pt]
        (10)--(12)--(11)--(13);
    \draw[line width=0.8pt]
        (8)--(13);

    \draw[line width=0.8pt]
        (1) .. controls +(-2,4) and +(3,4) .. (12);

\end{tikzpicture}
\caption{A crystal that is not an odd crystal (the black vertex has even degree)}
\end{subfigure}

\caption{Examples of crystals}
\label{fig:crystal}
\end{figure}
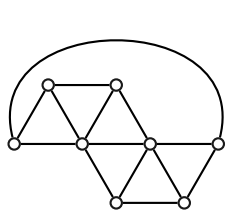
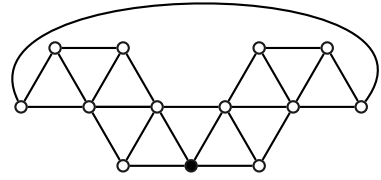

The class of triangularly connected graphs and the class of graphs containing
a spanning triangle-tree are incomparable; see
Section~\ref{sec:concludingremarks}.
Fan et al.~\cite{Fan2008} and Li et al.~\cite{LLW20} characterized the graphs
admitting nowhere-zero $3$-flows in these two classes, respectively. We extend
both characterizations to $S^1$-flows.

\begin{theorem}\label{thm:s1-characterization}
Let $G$ be a bridgeless graph.
\begin{enumerate}[label=(\arabic*)]
    \item \label{thm:s1-characterization:triangular}
    If $G$ is triangularly connected, then $G$ admits an $S^1$-flow if and
    only if $G$ is not an odd wheel.

    \item \label{thm:s1-characterization:crystal}
    If $G$ contains a spanning triangle-tree, then $G$ admits an $S^1$-flow if
    and only if $G$ is not an odd crystal.
\end{enumerate}
\end{theorem}

It is known that locally connected graphs, graph squares, $2$-connected
chordal graphs, triangulations of surfaces, and certain graph products are
triangularly connected; see, for example,~\cite{Imrich2010,Lai2003}.
Consequently, we obtain the following corollary.

\begin{corollary}
Let $G$ be a connected bridgeless graph that is not an odd wheel. Then  $G$ admits an $S^1$-flow if $G$ belongs to one
of the following classes:
\begin{enumerate}[label=(\arabic*)]
    \item \label{cor:s1-classes:square} the square of a connected graph with at least three vertices;
    \item \label{cor:s1-classes:locally-connected} a locally connected graph;
    \item \label{cor:s1-classes:chordal} a $2$-connected chordal graph.
\end{enumerate}
\end{corollary}

The remainder of the paper is organized as follows. In Section~2, we prove
Theorem~\ref{thm:degree4}. In Section~3, we introduce several reduction
techniques for $S^1$-flows. Sections~4 and~5 prove the two parts of
Theorem~\ref{thm:s1-characterization}, respectively. Finally, Section~6 compares
triangularly connected graphs with graphs containing a spanning triangle-tree
and proposes a conjecture concerning the relationship between $S^1$-flows and
nowhere-zero $4$-flows.

\section{$S^1$-flows and $3$-flows--Proof of Theorem~\ref{thm:degree4}}
In this section, we prove Theorem~\ref{thm:degree4}. 
We need the following elementary proposition.

\begin{proposition}
\label{prop:antipodal}
Let $\bm\alpha_1, \bm\alpha_2, \bm\alpha_3, \bm\alpha_4$ be four vectors in $\mathbb{R}^2$ with the same length. If $\bm\alpha_1 + \bm\alpha_2 + \bm\alpha_3 + \bm\alpha_4 = \bm0$, then they can be partitioned into two antipodal pairs. 
\end{proposition}

\begin{proof}
If the common length is zero, any pairing works. After scaling, assume that all four vectors are unit vectors, and put $\bm s=\bm\alpha_1+\bm\alpha_2=-(\bm\alpha_3+\bm\alpha_4)$. If $\bm s=\bm0$, then $\bm\alpha_2=-\bm\alpha_1$ and $\bm\alpha_4=-\bm\alpha_3$. Suppose that $\bm s\ne\bm0$. An unordered pair of unit vectors with prescribed nonzero sum $\bm s$ is unique: the two vectors have equal projection $\lVert\bm s\rVert/2$ in the direction of $\bm s$, and their perpendicular components are opposite and have a uniquely determined length. Hence $\{\bm\alpha_3,\bm\alpha_4\}=\{-\bm\alpha_1,-\bm\alpha_2\}$, which gives the
required two antipodal pairs.
\end{proof}

\noindent
\begin{proof}[{\bf Proof of Theorem~\ref{thm:degree4}.}]  By Theorem~\ref{thm:3-R3-S1-flow}, we only need to prove that  if $G$ admits an
$S^1$-flow, then it admits  a nowhere-zero $3$-flow.  Suppose otherwise, and choose a
counterexample $G$ with $|V_4(G)|$ minimum. By
Theorem~\ref{thm:3-R3-S1-flow}, $V_4(G)\ne\varnothing$.

Let $(D,\phi)$ be an $S^1$-flow of $G$, let $x$ be a $4$-vertex, and let
$e_1,e_2,e_3,e_4$ be its incident edges. Without loss of generality, assume that each $e_i$ is oriented toward $x$.    Let $\bm\alpha_i = \phi(e_i)$. Then
 \[
 \bm\alpha_1 + \bm\alpha_2 + \bm\alpha_3 + \bm\alpha_4 = \bm0.
 \]
 Since $\phi$ is an $S^1$-flow,  each $\bm\alpha_i$ is a unit vector. Thus by Proposition~\ref{prop:antipodal}, without loss of generality,  we may assume that $\bm\alpha_1 = -\bm\alpha_2$ and $\bm\alpha_3=-\bm\alpha_4$.  Let $G_1$ be the graph obtained from $G$ by splitting
$x$ into two $2$-vertices $x_1,x_2$, where $x_1$ is incident with
$e_1,e_2$ and $x_2$ with $e_3,e_4$.  The same orientation and edge values
form an $S^1$-flow of $G_1$. Since $|V_4(G_1)| = |V_4(G)| -1$, by the minimality of $G$, $G_1$ admits a nowhere-zero $3$-flow. Identifying $x_1$ with $x_2$, one can obtain a nowhere-zero $3$-flow of $G$ from any nowhere-zero $3$-flow of $G_1$, a contradiction. 
 \end{proof}

\section{Reduction operations for $S^1$-flows}
\label{sec:reduction}

In this section, we establish several reduction results for $S^1$-flows, including lemmas concerning 2- and 3-edge-cuts, $2$-sums, bull-growth operations, and wheel contractions.

\subsection{Small edge cuts and $S^1$-flow extensions at  $3^-$-vertices}

We first need some basic facts on the unit vectors in $\mathbb{R}^2$.
Throughout the paper, all vectors are treated as row vectors, and
$\bm a^{\mathsf T}$ denotes the transpose of a row vector $\bm a$.

\begin{proposition}
\label{prop:basic}
Let $\bm a,\bm b,\bm c$ be three vectors in $S^1$. We have the following:

{\makeatletter\def\@currentlabel{(1)}\label{prop:basic:1}\makeatother}
(1) There is an orthogonal matrix $A$ such that $\bm aA=\bm b$.

{\makeatletter\def\@currentlabel{(2)}\label{prop:basic:2}\makeatother}
(2) If $\bm a,\bm b,\bm c$ are the three cube roots of unity, then for any permutation $\bm a_1,\bm a_2,\bm a_3$ of $\bm a,\bm b,\bm c$, there is an orthogonal matrix $A$ such that $\bm aA=\bm a_1$, $\bm bA=\bm a_2$, and $\bm cA=\bm a_3$.
\end{proposition}

\begin{proof}
 \ref{prop:basic:1}   Extend $\bm a$ and $\bm b$ to orthonormal bases $\{\bm a,\bm a'\}$ and $\{\bm b,\bm b'\}$ of $\mathbb{R}^2$, respectively. The linear transformation mapping $\bm a$ to $\bm b$ and $\bm a'$ to $\bm b'$ is orthogonal. Hence there exists an orthogonal matrix $A$ such that $\bm aA=\bm b$.

\medskip
\ref{prop:basic:2} Denote $\bm a=(x,y)$, where $x^2+y^2=1$, and let
\[
R(\theta)=
\begin{pmatrix}
\cos\theta&\sin\theta\\
-\sin\theta&\cos\theta
\end{pmatrix},
\qquad
M=2\bm a^{\mathsf T}\bm a-I=
\begin{pmatrix}
2x^2-1&2xy\\
2xy&2y^2-1
\end{pmatrix}.
\]
Here $R(\theta)$ is the rotation through angle $\theta$, while $M$ is the reflection in the line through the origin in the direction of $\bm a$. Since $\bm a,\bm b,\bm c$ are the three cube roots of unity, $R(2\pi/3)$ and
$R(-2\pi/3)$ induce the two nontrivial cyclic permutations of $\bm a,\bm b,\bm c$.
Moreover, $M$ fixes $\bm a$ and interchanges $\bm b$ and $\bm c$. Consequently, the six matrices $R(2i\pi/3)$ and
$MR(2i\pi/3)$, where $i\in\{0,1,2\}$, realize all six permutations of
$\bm a,\bm b,\bm c$. Thus, for every permutation
$(\bm a_1,\bm a_2,\bm a_3)$ of $(\bm a,\bm b,\bm c)$, there is an orthogonal matrix $A$ such that
$\bm aA=\bm a_1$, $\bm bA=\bm a_2$, and $\bm cA=\bm a_3$.
\end{proof}

For a vertex set $X\subseteq V(G)$, let $\delta_G(X)$ denote the set of
edges with exactly one end in $X$; when $X=\{v\}$, we write
$\delta_G(v)$ for $\delta_G(\{v\})$. We omit the subscript when the graph
is clear. By Proposition~\ref{prop:basic}, we have the following lemma.

\begin{lemma}
\label{le:extension}
Let $G$ be a connected bridgeless graph that admits an $S^1$-flow $(D,\psi)$.

\begin{enumerate}[label=(\arabic*)]
\item \label{le:extension:1} For any edge $e$ in $G$ and any vector $\bm \alpha\in S^1$, $G$ has an $S^1$-flow $(D,\phi)$ such that $\phi(e)=\bm \alpha$.

\item \label{le:extension:2} Let $v$ be a $3^-$-vertex in $G$. For any
$\gamma:\delta(v)\to S^1$ satisfying
\[
\partial_D\gamma(v):=
\sum_{e\in E_D^+(v)}\gamma(e)-
\sum_{e\in E_D^-(v)}\gamma(e)=\bm0,
\]
there is an $S^1$-flow $(D,\phi)$ of $G$ such that
$\phi|_{\delta(v)}=\gamma$.
\end{enumerate}
\end{lemma}

\begin{proof}
\ref{le:extension:1} Let $\bm \beta=\psi(e)\in S^1$. By Proposition~\ref{prop:basic}\ref{prop:basic:1},
there is an orthogonal matrix $A$ such that $\bm \beta A=\bm \alpha$. Define
$\phi: E(G)\to S^1$ by $\phi(e')=\psi(e')A$ for every $e'\in E(G)$.
Since $A$ is orthogonal, $\phi(e')\in S^1$ for every $e'\in E(G)$.
Moreover, for every vertex $x\in V(G)$, we have
$\partial\phi(x)=(\partial \psi(x))A=\bm0$ by the linearity of the
transformation induced by $A$. Thus $(D,\phi)$ is an $S^1$-flow of $G$
with $\phi(e)=\bm \beta A=\bm \alpha$.

\medskip
\ref{le:extension:2}  If $d_G(v)=0$, there is nothing to prescribe. Since $G$ is bridgeless, we have $d_G(v)\ge2$, and hence $d_G(v)\in\{2,3\}$. Without loss of generality, we  assume that all
edges in $\delta(v)$ are directed away from $v$.

We first assume  $d_G(v)=2$. Let  $\delta(v)=\{e_1,e_2\}$. Then 
$\psi(e_1)+\psi(e_2) = \gamma(e_1)+\gamma(e_2)=\bm 0$. By
Proposition~\ref{prop:basic}\ref{prop:basic:1}, there is an orthogonal matrix $A$ such
that $\psi(e_1)A=\gamma(e_1)$. Define $\phi(e')=\psi(e')A$ for every
$e'\in E(G)$. Then   
$\phi(e_2)=\psi(e_2)A=-\psi(e_1)A=-\gamma(e_1)=\gamma(e_2)$.  Since $A$ is  orthogonal, $||\phi(e')||=||\psi(e')A||=||\psi(e')|| = 1$  for each edge $e'$.   Therefore,   $(D,\phi)$ is a desired $S^1$-flow of
$G$.

Now suppose that $d_G(v)=3$. Let
$\delta(v)=\{e_1,e_2,e_3\}$. Then 
 $\psi(e_1)+\psi(e_2)+\psi(e_3)=
\gamma(e_1)+\gamma(e_2)+\gamma(e_3)=\bm 0$. Since the sum of three unit
vectors in $\mathbb{R}^2$ is $\bm0$, each of the triples
$\psi(e_1),\psi(e_2),\psi(e_3)$ and $\gamma(e_1),\gamma(e_2),\gamma(e_3)$ consists
of the three cube roots of unity up to an orthogonal transformation.
Thus, by Proposition~\ref{prop:basic}\ref{prop:basic:1}, there are orthogonal matrices
$B$ and $C$ such that $\psi(e_i)B$ and $\gamma(e_i)C$, $i=1,2,3$, are the three cube roots of unity. By Proposition~\ref{prop:basic}\ref{prop:basic:2}, there is an orthogonal matrix $A$ such that $\psi(e_i)BA=\gamma(e_i)C$ for every $i\in\{1,2,3\}$. Define $\phi(e')=\psi(e')BAC^{-1}$ for every $e'\in E(G)$.   Then $\phi(e_i)=\gamma(e_i)$ for $i=1,2,3$.  Since $BAC^{-1}$ is orthogonal,  $||\phi(e')||=||\psi(e')BAC^{-1}||=||\psi(e')|| = 1$ for each edge $e'$. Therefore,  $(D,\phi)$ is a desired $S^1$-flow of $G$.
\end{proof} 
 
 For an edge $e$ of a graph $G$, we write $G/e$ for the graph obtained from $G$ by contracting the edge $e$, that is, by identifying the two endvertices of $e$ and then deleting all resulting loops. For a subgraph $H$ of $G$, let $G/H$ denote the graph obtained by contracting all edges of $H$ and then deleting the resulting loops.
 
 As a corollary of Lemma~\ref{le:extension}, we have the following lemma  indicating that small edge-cuts are reducible.
 
\begin{lemma}\label{lem:small-cut-s1-flow}
Let $G$ be a $2$-edge-connected graph, and let $M=\delta_G(X)$ with
$2 \leq |M| \leq 3$. Write $G_1=G[X]$ and $G_2=G-X$. Then $G$ admits an
$S^1$-flow if and only if both $G/G_1$ and $G/G_2$ admit $S^1$-flows.
\end{lemma}

 \begin{proof}
Clearly, both $G_1$ and $G_2$ are connected.
Since contraction preserves flows, if $G$ admits an $S^1$-flow, then both $G/G_1$ and $G/G_2$ admit $S^1$-flows.

Now assume that both $G/G_1$ and $G/G_2$ admit $S^1$-flows. Let $H_1=G/G_2$ and $H_2=G/G_1$, and for $i\in\{1,2\}$, let $v_i$ be the vertex of $H_i$ obtained by contracting $G_{3-i}$. Note that $\delta_{H_1}(v_1)=\delta_{H_2}(v_2)=M$.  Let $(D_1,\phi_1)$ be an $S^1$-flow of $H_1$. Without loss of generality, assume that  all edges in $\delta_{H_1}(v_1)$ are oriented toward $v_1$ and  thus $\sum_{e\in M}\phi_1(e)=\bm0$.

Define $\gamma: M\to S^1$ by $\gamma(e)=\phi_1(e)$ for every $e\in M$.  Let $D_2$ be an orientation of $H_2$ such that all edges in $\delta_{H_2}(v_2)$ are oriented  away  from $v_2$. Then
$\partial_{D_2}\gamma(v_2)=\bm0$. Since $G$ is $2$-edge-connected,
$H_2$ is bridgeless. Since $2\le |M|\le3$, we have
$2\le d_{H_2}(v_2)=|M|\le3$. By Lemma~\ref{le:extension}\ref{le:extension:2}, $H_2$ admits
an $S^1$-flow $(D_2,\phi_2)$ such that
$\phi_2(e)=\gamma(e)=\phi_1(e)$ for every $e\in M$.

Therefore, the two $S^1$-flows $(D_1,\phi_1)$ and $(D_2, \phi_2)$ can be combined into an $S^1$-flow of $G$.
\end{proof}

\subsection{$2$-sums  and $S^1$-flows}
We next establish reduction results based on the following operation, called a
``$2$-sum''.

\begin{definition}
For two graphs $G_1$ and $G_2$, we say that $G$ is the
\emph{$2$-sum} of $G_1$ and $G_2$ along an edge $e=uv$, denoted by
$G=G_1\oplus_e G_2$ (or simply $G=G_1\oplus_2 G_2$ when the edge $e$ is
clear from the context), if
$E(G)=E(G_1)\cup E(G_2)$,
$E(G_1)\cap E(G_2)=\{e\}$, and
$V(G_1)\cap V(G_2)=\{u,v\}$.
\end{definition}

Informally, the $2$-sum operation can be viewed as gluing two graphs along
the edge $e$. For example, a $2$-sum of two copies of $K_4$ is illustrated in
\Cref{2sum}.

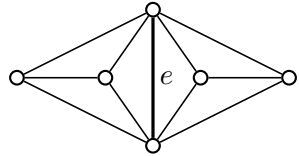
\begin{figure}[htbp]  
\centering
\begin{tikzpicture}[
    >=latex,
    roundnode/.style={
        circle,
        draw=black,
        thick,
        minimum size=1.8mm,
        inner sep=0pt,
        fill=white
    },
    scale=0.45
]

\node[roundnode] (1) at (-4,0) {};
\node[roundnode] (2) at (0,-2) {};
\node[roundnode] (3) at (0,2) {};
\node[roundnode] (4) at (-1.4,0) {};
\node[roundnode] (5) at (1.4,0) {};
\node[roundnode] (6) at (4,0) {};

\draw[line width=0.6pt]
(1)--(2)--(6)--(3)--(1)--(4)--(2)--(5)--(6);

\draw[line width=0.6pt]
(4)--(3);

\draw[line width=1.2pt]
(2)--(3);

\draw[line width=0.6pt]
(3)--(5);

\node at (0.4,0) {$e$};

\end{tikzpicture}
\caption{$K_4\oplus_e K_4$}
\label{2sum}
\end{figure}

 We give the following theorem for constructing $S^1$-flows from $2$-sums.

\begin{theorem}
\label{thm:intro-2-sum-general}
Let $G=G_1\oplus_e G_2$. 

\begin{enumerate}[label=(\roman*)]
\item \label{thm:intro-2-sum-general:i} If both $G_1$ and $G_2$  admit $S^1$-flows, then both $G$ and $G-e$ admit $S^1$-flows.

\item \label{thm:intro-2-sum-general:ii} If  $G_1/e$ admits a nowhere-zero $3$-flow and $G_2$ admits an $S^1$-flow, then  $G$ admits an $S^1$-flow.

\item \label{thm:intro-2-sum-general:iii} If $G_1$ and $G_2-e$ both admit  $S^1$-flows, then $G$ admits an $S^1$-flow.

\item \label{thm:intro-2-sum-general:iv} If $G_1$ and $G_2$ are bridgeless and both $G_1/e$ and $G_2/e$ admit nowhere-zero $3$-flows,  then $G$ admits an $S^1$-flow.
\end{enumerate}
\end{theorem}

A \emph{digon} is a graph consisting of two vertices joined by two parallel edges. The following two results are needed to prove Theorem~\ref{thm:intro-2-sum-general}.

\begin{theorem}[\cite{DeVosXY2006,Lai2000}]
\label{thm:Z3-connected}  Let $G$ be a graph and $H$ be a subgraph of $G$. If $H$ is a digon or an even wheel, then 
 $G$ admits a nowhere-zero $3$-flow if and only if $G/H$  admits a nowhere-zero $3$-flow.
\end{theorem}

\begin{definition}
Let $H$ be a graph with an orientation $D$, and let $A$ be either an
abelian group or a real vector space. An assignment $F:E(H)\to A$ is called
an \emph{$A$-preflow} if $\partial F(v)=0$ for every vertex $v$ except
possibly a specified set of vertices. It is \emph{nowhere-zero} if
$F(e)\ne0$ for every edge $e$. When $A=\mathbb R^2$ and
$F(e)\in S^1$ for every edge, we call $F$ an \emph{$S^1$-preflow}.
\end{definition}

The following lemma provides the key preflow construction for obtaining
$S^1$-flows from $2$-sums.

\begin{lemma}[Two-Terminal Preflow Lemma]
\label{lem:sqrt3-preflow}
Let $H$ be a  graph admitting a nowhere-zero $3$-flow. For any two
distinct vertices $p,q\in V(H)$  contained in the same connected component and any unit vector $\mathbf{u}\in\mathbb{R}^2$,
there exists an $S^1$-preflow $F$ such that
$\partial F(v)=\bm{0}$ for all $v\in V(H)\setminus\{p,q\}$,
$\partial F(p)=-\sqrt{3}\,\mathbf{u}$, and
$\partial F(q)=\sqrt{3}\,\mathbf{u}$.
\end{lemma}

\begin{proof}
On every component other than the one containing $p$ and $q$, choose an
$R_3$-flow supplied by Theorem~\ref{thm:3-R3-S1-flow}; this is already an
$S^1$-flow. It therefore suffices to treat the case in which $H$ is
connected. Let $(D,\varphi)$ be a nowhere-zero $3$-flow of $H$.
By reversing edges with negative values, we may assume that
$\varphi(e)\in\{1,2\}$ for every edge $e$.  Then $D$ is strongly connected and hence, there exists a directed path $P$ from $p$ to $q$.  Define
\[
\varphi'(e)=
\begin{cases}
\varphi(e)-3, & \text{if }e\in E(P),\\
\varphi(e), & \text{if }e\notin E(P).
\end{cases}
\]
Then, 
$\partial_D\varphi'(v)=0$ for all $v\notin\{p,q\}$, while
$\partial_D\varphi'(p)=-3$ and $\partial_D\varphi'(q)=3$.
Since $\varphi(e)\in\{1,2\}$, we have
$\varphi'(e)\in\{-2,-1,1,2\}$ for every edge $e$. 

Let $H'$ be the spanning subgraph whose edges satisfy
$|\varphi'(e)|=1$. Reducing the boundary equations modulo $2$ shows that
$p$ and $q$ are the only odd-degree vertices of $H'$. Thus, $p$ and $q$ lie in the same component of $H'$. Take an Euler trail from
$q$ to $p$ in that component and an Euler tour in every other nontrivial
component. Orient each trail or tour according to
its traversal. Define $\psi\colon E(H)\to\{-1,0,1\}$ by setting $\psi (e)=0$ for $e\notin E(H')$, and, for $e\in E(H')$, setting $\psi(e)=1$ or $-1$ according as the traversal direction agrees or disagrees with $D$. Then $\partial_D\psi(p)=-1$, $\partial_D\psi(q)=1$, and $\partial_D\psi(v)=0$ for every other vertex $v$. Consequently,
$\partial_D\varphi'=3\partial_D\psi$.

Now define
\[
f_1=-\psi,\qquad
f_2=\frac{\varphi'+\psi}{2},\qquad
f_3=\frac{\psi-\varphi'}{2}.
\]
Since $\varphi'$ and $\psi$ have the same parity on every edge,
$f_1,f_2,f_3$ are integer-valued.

Moreover, for every edge $e$, exactly two of
$f_1(e),f_2(e),f_3(e)$ are nonzero, and these two values are opposites.
More precisely,
\begin{center}
\begin{tabular}{c|c}
$(\varphi'(e),\psi(e))$ & $(f_1(e),f_2(e),f_3(e))$ \\
\hline
$\psi(e)=\varphi'(e)=\pm1$
    & $(-\varphi'(e),\varphi'(e),0)$ \\
$\psi(e)=-\varphi'(e)=\pm1$
    & $(\varphi'(e),0,-\varphi'(e))$ \\
$\psi(e)=0,\ \varphi'(e)=\pm2$
    & $(0,\varphi'(e)/2,-\varphi'(e)/2)$
\end{tabular}
\end{center}

Consider an equilateral triangle in $\mathbb{R}^2$ with side length $1$ and
center at the origin, and let its vertices be $\bm a_1,\bm a_2,\bm a_3$. Then
$\bm a_1+\bm a_2+\bm a_3=\bm{0}$, $\|\bm a_i-\bm a_j\|=1$ for $i\neq j$, and
$\|\bm a_i\|=1/\sqrt{3}$.

For each edge $e$, define
$F(e)=f_1(e)\bm a_1+f_2(e)\bm a_2+f_3(e)\bm a_3$.
By the preceding table, exactly two of the three coefficients are nonzero
and are opposites of each other. Thus $F(e)$ is of the form
$\pm(\bm a_i-\bm a_j)$, and hence $\|F(e)\|=1$. Therefore, $F$ is an
$S^1$-preflow.

By $\partial\varphi'=3\partial\psi$ and the definitions of
$f_1,f_2,f_3$, we have
$\partial f_1=-\partial\psi$,
$\partial f_2=2\partial\psi$, and
$\partial f_3=-\partial\psi$. Therefore,
$\partial F=(-\bm a_1+2\bm a_2-\bm a_3)\partial\psi=3\bm a_2\partial\psi$ since $\bm a_1+\bm a_2+\bm a_3=\bm{0}$. It follows that
$\partial F(p)=-3\bm a_2$, $\partial F(q)=3\bm a_2$, and
$\partial F(v)=\bm{0}$ for every $v\in V(H)\setminus\{p,q\}$.

Since $\|3\bm a_2\|=\sqrt{3}$, the vector
$\mathbf{u}_0=\sqrt{3}\,\bm a_2$ is a unit vector in $\mathbb{R}^2$. Hence,
$\partial F(p)=-\sqrt{3}\,\mathbf{u}_0$ and
$\partial F(q)=\sqrt{3}\,\mathbf{u}_0$. Applying a common rotation of
$\mathbb{R}^2$ to all edge values of $F$ replaces $\mathbf{u}_0$ by the
prescribed unit vector $\mathbf{u}$. This completes the proof.
\end{proof}

Now we are ready to prove Theorem~\ref{thm:intro-2-sum-general}.

\begin{proof}[\bf Proof of Theorem~\ref{thm:intro-2-sum-general}] Let $u,v$ be the endvertices of $e$ in both $G_1$ and $G_2$.

\medskip \noindent
\ref{thm:intro-2-sum-general:i} Let $(D_1,\phi_1)$ be an $S^1$-flow of $G_1$, orient $e$ from $u$
to $v$, and write $\phi_1(e)=\bm a$. Choose $\bm b\in S^1$ so that
$\bm c:=\bm a+\bm b\in S^1$.
Let $D_2$ be an orientation of $G_2$  such that $e$  is oriented also from $u$ to $v$.  By Proposition~\ref{prop:basic}\ref{prop:basic:1}, $G_2$ has two $S^1$-flows $(D_2, \phi_2)$ and $(D_2, \phi_3)$ such that $\phi_2(e) =-\bm a$ and $\phi_3(e) = \bm b$.

Therefore, one can obtain an $S^1$-flow of $G-e$ by combining $(D_1, \phi_1)$ and $(D_2, \phi_2)$ and an $S^1$-flow of $G$ by  combining $(D_1, \phi_1)$ and $(D_2, \phi_3)$.
 
\medskip \noindent
\ref{thm:intro-2-sum-general:ii}  Let $G_1'$ be the graph obtained from $G_1$ by adding a  new edge  $f$ parallel to $e$.  Label the edge $e$ in $G_2$ as $f$. Then $G = (G_1'\oplus_f G_2)-f$.  Since $G_1/e= G_1'/\{e,f\}$ and $G_1/e$ admits a nowhere-zero $3$-flow,  by Theorem~\ref{thm:Z3-connected},  $G_1'$ also admits a nowhere-zero $3$-flow, and hence an $S^1$-flow by Theorem~\ref{thm:3-R3-S1-flow}. Note that both $G_1'$ and $G_2$ admit $S^1$-flows. By \ref{thm:intro-2-sum-general:i}, $G= (G_1'\oplus_fG_2)-f$ admits an $S^1$-flow.

\medskip \noindent
\ref{thm:intro-2-sum-general:iii} Since $G$ is the edge-disjoint union of $G_1$ and $G_2-e$, we can obtain an $S^1$-flow on $G$ by combining an $S^1$-flow on $G_1$ with an $S^1$-flow on $G_2-e$.

\medskip \noindent
\ref{thm:intro-2-sum-general:iv} If one of $G_1,G_2$ admits an $S^1$-flow, then \ref{thm:intro-2-sum-general:iv} follows from \ref{thm:intro-2-sum-general:ii}. 

Now 
assume that neither graph admits an $S^1$-flow; in particular, neither
admits a nowhere-zero $3$-flow. For each $i\in\{1,2\}$, since $G_i/e$ admits a nowhere-zero $3$-flow but $G_i$ does not, by Tutte's flow polynomial theorem  \cite{Tutte54}(see Lemma 2.7.1 (1) in  \cite{zhang1997integer}),  $G_i-e$ admits a nowhere-zero $3$-flow. 

For each $i\in\{1,2\}$, let $D_i$ be an orientation of $G_i$ such that $e$ is oriented from $u$ to $v$ in both $D_i$. Choose $\bm a_1,\bm a_2\in S^1$ with
$\bm a_1\mathbin{\cdot}\bm a_2=-5/6$, and set
\(
\bm c=-\sqrt3(\bm a_1+\bm a_2).
\)
Then $\lVert\bm c\rVert^2=3(2-5/3)=1$, so $\bm c\in S^1$.
Since $G_i$ is bridgeless, $u$ and $v$ lie in the same component of
$G_i-e$.
By Lemma~\ref{lem:sqrt3-preflow},  each $G_i -e$  has an $S^1$-preflow $(D_i, \phi_i)$ such that  $\partial \phi_i(x) =0$ if $x \not \in \{u,v\}$, 
$\partial \phi_i(u)=\sqrt{3}\bm a_i$, and $\partial \phi_i(v)=-\sqrt{3}\bm a_i$. Define $\phi: E(G) \to S^1$ as follows:
 \[
\phi(e')=
\begin{cases}
\phi_1(e'), & \text{if } e' \in E(G_1)\setminus\{e\},\\
\phi_2(e'), & \text{if } e' \in E(G_2)\setminus\{e\},\\
\bm c, & \text{if } e'=e.
\end{cases}
\]
Let $D$ be the orientation of $G$ whose restriction to each $G_i-e$ is
$D_i$ and in which $e$ is directed from $u$ to $v$. At $u$ we have
\[
\partial\phi(u)=\sqrt3\bm a_1+\sqrt3\bm a_2+\bm c=\bm0,
\]
and the equation at $v$ is its negative; all other boundary equations hold
by construction. Thus $(D,\phi)$ is an $S^1$-flow of $G$.
\end{proof}

\medskip 
Li et al.~\cite{LLW20} introduced the following bull-growth operation in the
study of $3$-flows, which is closely related to the $2$-sum. We adapt this
operation to the setting of $S^1$-flows.

\begin{definition}\label{def:bull-reduction}
\textup{(\cite{LLW20})}
Let $u$ and $v$ be adjacent $3$-vertices of a graph $G$, each having
three distinct neighbors, with
$N_G(u)=\{v,w,a\}$ and $N_G(v)=\{u,w,b\}$. The vertices $a$ and $b$
may coincide. Define
\[
G_1=
\begin{cases}
G-u-v+ab, & \text{if } a\neq b,\\
G-u-v, & \text{if } a=b.
\end{cases}
\]
Then $G_1$ is called the \emph{bull-reduction} of $G$, and $G$ is called the
\emph{bull-growth} of $G_1$ (see \Cref{fig:bull-reduction}). We write
$G=\mathcal{B}\uplus G_1$.
\end{definition}

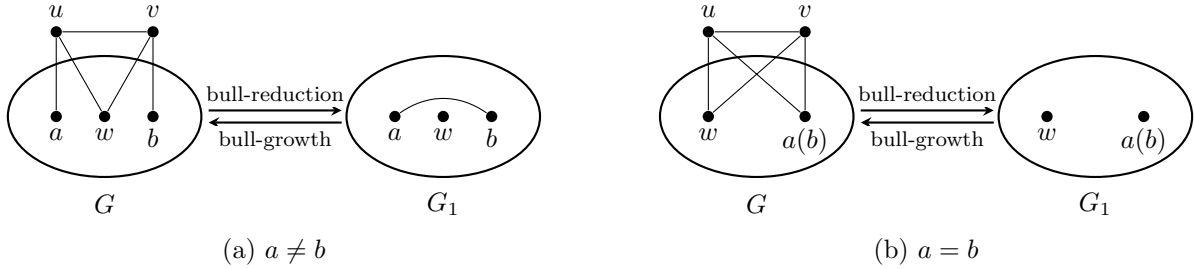
\begin{figure}[htbp]
\centering

\begin{subfigure}{0.48\textwidth}
\centering
\begin{tikzpicture}[
    scale=0.8,
    vertex/.style={circle,fill=black,inner sep=1.5pt},
    label_style/.style={fill=none,draw=none,inner sep=1pt,font=\small},
    operation/.style={fill=white,draw=none,inner sep=2pt,font=\small},
    graphboundary/.style={thick},
    operationarrow/.style={thick,>=stealth}
]


\node[vertex] (u) at (-0.8,1.4) {};
\node[label_style, above=2pt of u] {$u$};

\node[vertex] (v) at (0.8,1.4) {};
\node[label_style, above=2pt of v] {$v$};

\node[vertex] (w) at (0,0) {};
\node[label_style, below=1pt of w] {$w$};

\node[vertex] (a) at (-0.8,0) {};
\node[label_style, below=1pt of a] {$a$};

\node[vertex] (b) at (0.8,0) {};
\node[label_style, below=1pt of b] {$b$};

\draw
(u)--(v)
(u)--(a)
(u)--(w)
(v)--(w)
(v)--(b);

\draw[graphboundary] (0,0) ellipse (1.6 and 1);

\node[label_style] at (0,-1.45) {$G$};


\draw[operationarrow,->]
(1.72,0.1)--(3.92,0.1)
node[midway,above=1pt,operation]{\scriptsize bull-reduction};

\draw[operationarrow,<-]
(1.72,-0.1)--(3.92,-0.1)
node[midway,below=1pt,operation]{\scriptsize bull-growth};


\begin{scope}[xshift=5.6cm,yshift=0cm]

\node[vertex] (a1) at (-0.8,0) {};
\node[label_style, below=1pt of a1] {$a$};

\node[vertex] (w1) at (0,0) {};
\node[label_style, below=1pt of w1] {$w$};

\node[vertex] (b1) at (0.8,0) {};
\node[label_style, below=1pt of b1] {$b$};

\draw (a1) to[bend left=35] (b1);

\draw[graphboundary] (0,0) ellipse (1.6 and 1);

\node[label_style] at (0,-1.45) {$G_1$};

\end{scope}

\end{tikzpicture}

\caption{$a\neq b$}
\label{fig:bull-a-ne-b}
\end{subfigure}
\hfill
\begin{subfigure}{0.48\textwidth}
\centering
\begin{tikzpicture}[
    scale=0.8,
    vertex/.style={circle,fill=black,inner sep=1.5pt},
    label_style/.style={fill=none,draw=none,inner sep=1pt,font=\small},
    operation/.style={fill=white,draw=none,inner sep=2pt,font=\small},
    graphboundary/.style={thick},
    operationarrow/.style={thick,>=stealth}
]


\node[vertex] (u) at (-0.8,1.4) {};
\node[label_style, above=2pt of u] {$u$};

\node[vertex] (v) at (0.8,1.4) {};
\node[label_style, above=2pt of v] {$v$};

\node[vertex] (w) at (-0.8,0) {};
\node[label_style, below=1pt of a] {$w$};

\node[vertex] (b) at (0.8,0) {};
\node[label_style, below=1pt of b] {$a(b)$};

\draw
(u)--(v)
(u)--(w)
(u)--(b)
(v)--(w)
(v)--(b);

\draw[graphboundary] (0,0) ellipse (1.6 and 1);

\node[label_style] at (0,-1.45) {$G$};


\draw[operationarrow,->]
(1.72,0.1)--(3.92,0.1)
node[midway,above=1pt,operation]{\scriptsize bull-reduction};

\draw[operationarrow,<-]
(1.72,-0.1)--(3.92,-0.1)
node[midway,below=1pt,operation]{\scriptsize bull-growth};


\begin{scope}[xshift=5.6cm,yshift=0cm]

\node[vertex] (a1) at (-0.8,0) {};
\node[label_style, below=1pt of a1] {$w$};

\node[vertex] (b1) at (0.8,0) {};
\node[label_style, below=1pt of b1] {$a(b)$};

\draw[graphboundary] (0,0) ellipse (1.6 and 1);

\node[label_style] at (0,-1.45) {$G_1$};

\end{scope}

\end{tikzpicture}

\caption{$a=b$}
\label{fig:bull-a-=-b}
\end{subfigure}

\caption{Bull-reduction and bull-growth.}
\label{fig:bull-reduction}
\end{figure}

The two cases in the definition of the bull-growth operation can be described in terms of edge sums as follows. Suppose first that $a\neq b$, and let $B_3$ be the graph with
\[
V(B_3)=\{a,b,u,v,w'\}
\quad\text{and}\quad
E(B_3)=\{uv,uw',w'v,ua,ab,bv\},
\]
where $w'$ is a new vertex. Thus, $B_3$ is the $2$-sum of a triangle and a $4$-cycle, and hence admits a nowhere-zero $3$-flow. The graph $G$ is obtained from
$\bigl(G_1\oplus_{ab}B_3\bigr)-ab$ by identifying $w'$ with $w$.
We will use the notation $B_3$ for this graph throughout the proof. 

If \(a=b\), let \(K_4\) be the complete graph on
\(\{a,u,v,w\}\). In this case,
\[
G=\bigl(K_4\oplus_{aw}(G_1+aw)\bigr)-aw.
\]

These descriptions, together with
Theorem~\ref{thm:intro-2-sum-general}, yield the following result.

\begin{theorem}
\label{thm:intro-bull-growth}
If \(G_1\) admits an \(S^1\)-flow, then every graph $G=\mathcal{B}\uplus G_1$ obtained from \(G_1\) by a bull-growth operation also admits an $S^1$-flow.
\end{theorem}

\begin{proof}
Suppose first that \(a=b\).  Then $G$ can be decomposed into two edge-disjoint subgraphs:  \(K_4-aw\) and $G_1$, each of which admits an $S^1$-flow.  Thus $G$ admits an $S^1$-flow.

Now suppose that \(a\neq b\). The graph \(B_3\) admits a nowhere-zero
\(3\)-flow and hence an \(S^1\)-flow. It follows from
Theorem~\ref{thm:intro-2-sum-general}\ref{thm:intro-2-sum-general:i} that
\(
J:=\bigl(G_1\oplus_{ab}B_3\bigr)-ab
\)
admits an \(S^1\)-flow.
The graph \(G\) is obtained from \(J\) by identifying \(w'\) with \(w\) and thus an $S^1$-flow of $J$  induces an $S^1$-flow of $G$.
\end{proof}

\subsection{Contractible configurations for  \(S^1\)-flows}

Identifying contractible configurations is a powerful technique in the study of flows. It is well known that even wheels are  $3$-flow contractible, whereas odd wheels are not. In this subsection, we show that both odd and even wheels satisfy an analogous contraction property for \(S^1\)-flows.

Recall that the wheel \(W_k\) is obtained from a cycle \(C_k\) by adding a new vertex adjacent to every vertex of \(C_k\). The cycle \(C_k\) is called the \emph{rim}, and the new vertex is called the \emph{hub}. The vertices and edges of \(C_k\) are called the \emph{rim vertices} and \emph{rim edges}, respectively, while the edges incident with the hub are called the \emph{spokes}. Whenever \(W_k\) is mentioned, we assume that \(k\geq 3\). We write \(\Z_k\) for the cyclic group of order \(k\).

Our main result in this subsection shows that every wheel is, in an appropriate sense, \(S^1\)-flow-contractible.

\begin{theorem}
\label{thm:intro-wheel-contraction-general}
Let \(G\) be a \(2\)-connected graph containing a wheel \(W_n\) as a proper subgraph. If \(G/W_n\) admits a nowhere-zero \(3\)-flow, then \(G\) admits an \(S^1\)-flow.
\end{theorem}

The proof of Theorem~\ref{thm:intro-wheel-contraction-general} requires the following result for odd wheels.

\begin{lemma}[\cite{Xu2004}]
\label{lem:extended-z3-boundary}
Let \(n\geq 3\) be odd, and let \(W_n\) be a wheel with hub \(h\) and rim
vertices \(v_0,v_1,\ldots,v_{n-1}\). Let
\(\beta:V(W_n)\to\Z_3\) be a function satisfying
\(\sum_{v\in V(W_n)}\beta(v)=0\).
If \(\beta\) is not identically zero, then \(W_n\) has a preflow
\((D,\phi)\) with \(\phi:E(W_n)\to\mathbb Z_3\setminus\{0\}\) such that
\(\partial\phi(v)=\beta(v)\) for every \(v\in V(W_n)\).
\end{lemma}

Tutte~\cite{Tutte54} proved that a graph admits a nowhere-zero \(k\)-flow if and only if it admits a nowhere-zero \(\Z_k\)-flow. Combining this equivalence with Lemma~\ref{lem:extended-z3-boundary} gives the following consequence.

\begin{lemma}
\label{lem:wn}
Let \(G\) contain an induced odd wheel \(W_n\) as a subgraph. Suppose that \(G/W_n\) admits a nowhere-zero \(\Z_3\)-flow, whereas \(G\) does not. Then \(G-E(W_n)\) admits a nowhere-zero \(3\)-flow.
\end{lemma}

\begin{proof}
Let \(\psi\) be a nowhere-zero \(\Z_3\)-flow of \(G/W_n\), and lift its orientations and values to the edges of \(G-E(W_n)\). For each \(v\in V(W_n)\), let \(\alpha(v)\) denote the boundary of this lifted assignment at \(v\). Since \(\psi\) satisfies the flow equation at the vertex obtained by contracting \(W_n\), we have \(\sum_{v\in V(W_n)}\alpha(v)=0\).

Suppose that \(\alpha\) is not identically zero. Applying Lemma~\ref{lem:extended-z3-boundary} to the boundary function \(-\alpha\), we obtain a preflow \(\phi\) of \(W_n\) satisfying \(\partial\phi(v)=-\alpha(v)\) for every \(v\in V(W_n)\). The assignments \(\psi\) and \(\phi\) then combine to form a nowhere-zero \(\Z_3\)-flow of \(G\), contrary to the hypothesis.

Therefore, \(\alpha\) is identically zero. Hence the lifted assignment is itself a nowhere-zero \(\Z_3\)-flow of \(G-E(W_n)\). Tutte's equivalence now implies that \(G-E(W_n)\) admits a nowhere-zero \(3\)-flow.
\end{proof}

Although an odd wheel does not admit a nowhere-zero \(3\)-flow, the following lemma shows that it admits sufficiently flexible \(S^1\)-preflows. This result is analogous to Lemma~\ref{lem:sqrt3-preflow}.

\begin{lemma}
\label{lem:odd-wheel-any-absorber}
Let \(W_n\) be an odd wheel with \(n\geq 3\). For any two distinct vertices
\(p,q\in V(W_n)\) and any vector \(\bm b\in\mathbb R^2\) with
\(\lVert\bm b\rVert=\sqrt{3}\), there exists an \(S^1\)-preflow \(F\) on
\(W_n\) such that \(\partial F(p)=-\bm b\), \(\partial F(q)=\bm b\), and
\(\partial F(v)=\bm 0\) for every \(v\in V(W_n)\setminus\{p,q\}\).
\end{lemma}

\begin{proof}
Let
\[
R=
\begin{pmatrix}
1/2 & \sqrt3/2\\
-\sqrt3/2 & 1/2
\end{pmatrix}
\]
be the rotation through angle $\pi/3$. Then $R$ is orthogonal, $R^3=-I$, and
for every $\bm x\in S^1$, we have
$\lVert \bm x-\bm xR\rVert=\lVert \bm x-\bm xR^{-1}\rVert=1$.

Choose $\bm a\in\mathbb R^2$ such that $\lVert\bm a\rVert=\sqrt3$ and
$\lVert\bm a-\bm b\rVert=1$. Such a vector exists because there is a
triangle with side lengths $\sqrt3,\sqrt3,1$. Set
$\bm c=\bm a-\bm b$, so $\bm c\in S^1$. Since $R^{-1}$ is a rotation through angle $-\pi/3$, the vector
$\bm d+\bm dR^{-1}$ is obtained from $\bm d$ by rotating it through
angle $-\pi/6$ and multiplying its length by $\sqrt3$. Therefore, for the
given vector $\bm a$ with $\lVert\bm a\rVert=\sqrt3$, there is a unique
$\bm d\in S^1$ such that
$\bm a=\bm d+\bm dR^{-1}$.  

Let $h$ be the hub of $W_n$, and label its rim vertices cyclically as
$v_0,v_1,\ldots,v_{n-1}$, where all indices are taken modulo $n$.
Orient each rim edge $v_iv_{i+1}$ from $v_i$ to $v_{i+1}$ and each spoke
$hv_i$ away from $h$. 
For an edge assignment $F$ under this orientation,
write $r_i=F(v_iv_{i+1})$ and $s_i=F(hv_i)$. Then
$\partial F(v_i)=r_i-r_{i-1}-s_i$.

\medskip
\noindent\textbf{Case 1: Both $p$ and $q$ are rim vertices.}

Without loss of generality, assume that $p=v_0$ and $q=v_t$, where $t$ is odd and $1\le t\le n-2$.
If $t=1$, the sequence $\varepsilon_1,\ldots,\varepsilon_{t-1}$ below is
understood to be empty.

Since $t-1$ is even, choose
$\varepsilon_1,\ldots,\varepsilon_{t-1}\in\{1,-1\}$ such that
$\sum_{i=1}^{t-1}\varepsilon_i=0$. Since $n-t-1$ is odd, choose
$\eta_{t+1},\ldots,\eta_{n-1}\in\{1,-1\}$ such that
$\sum_{i=t+1}^{n-1}\eta_i=1$.
Define the values on the rim edges by
\[
r_0=\bm c,\qquad
r_i=r_{i-1}R^{\varepsilon_i}
\quad (1\le i\le t-1),
\]
and
\[
r_t=\bm dR^{-1},\qquad
r_i=r_{i-1}R^{\eta_i}
\quad (t+1\le i\le n-1).
\]
Then $r_{t-1}=\bm c$ and $r_{n-1}=\bm d$.
Define the values on the spokes by
\[
s_0=\bm dR^{-1},\qquad
s_t=-\bm d,\qquad
s_i=r_i-r_{i-1}\quad (i\notin\{0,t\}).
\]
For $i\notin\{0,t\}$, we have $r_i=r_{i-1}R^{\pm1}$, and hence
$\lVert s_i\rVert=\lVert r_i-r_{i-1}\rVert=1$. Thus all rim edges and
spokes receive values in $S^1$.

For every $i\notin\{0,t\}$, we have
$\partial F(v_i)=r_i-r_{i-1}-s_i=\bm0$. At $v_0$, we have
\[
\partial F(v_0)
=r_0-r_{n-1}-s_0
=\bm c-\bm d-\bm dR^{-1}
=\bm c-\bm a
=-\bm b,
\]
whereas at $v_t$,
\[
\partial F(v_t)
=r_t-r_{t-1}-s_t
=\bm dR^{-1}-\bm c+\bm d
=\bm a-\bm c
=\bm b.
\]
Finally, since the sum of the boundaries over all vertices is zero, we have
$\partial F(h)=\bm0$. Therefore, $F$ is a desired $S^1$-preflow.

\medskip
\noindent\textbf{Case 2: One of $p$ and $q$ is the hub.}

Without loss of generality, 
assume that $p=h$ and $q=v_0$. Since $n-1$ is even, choose
$\varepsilon_1,\ldots,\varepsilon_{n-1}\in\{1,-1\}$ such that
$\sum_{i=1}^{n-1}\varepsilon_i=2$.

Define the rim values by
\[
r_0=\bm d,\qquad
r_i=r_{i-1}R^{\varepsilon_i}
\quad (1\le i\le n-1).
\]
Then $r_{n-1}=\bm dR^2=-\bm dR^{-1}$, since $R^3=-I$.

Define the spoke values by $s_0=\bm c$ and
$s_i=r_i-r_{i-1}$ for $1\le i\le n-1$. As in Case~1, all these values
belong to $S^1$, and $\partial F(v_i)=\bm0$ for $1\le i\le n-1$.

At $v_0$, we have
\[
\partial F(v_0)
=r_0-r_{n-1}-s_0
=\bm d+\bm dR^{-1}-\bm c
=\bm a-\bm c
=\bm b.
\]
It follows that $\partial F(h)=-\sum_{i=0}^{n-1}\partial F(v_i)=-\bm b$. Hence $F$ is again a desired
$S^1$-preflow.
\end{proof}

We are now ready to prove
Theorem~\ref{thm:intro-wheel-contraction-general}.

\begin{proof}[\bf Proof of Theorem~\ref{thm:intro-wheel-contraction-general}]

If $n$ is even, by Theorem~\ref{thm:Z3-connected},  $G$ admits a nowhere-zero $3$-flow and thus an $S^1$-flow by  Theorem~\ref{thm:3-R3-S1-flow}. Now  we assume that $n$ is odd.

 Let $H=G[V(W_n)]$ be the subgraph of \(G\) induced by the vertices of \(W_n\). If $W_n$   is not an induced subgraph of $G$, then any additional edge with both
ends in $V(W_n)$ either forms a digon with an edge of $W_n$ or, if it is a chord
of the rim, together with the corresponding spokes forms an even wheel. Thus $G/H$ can be obtained by successively contracting digons or by contracting an even wheel and then successively contracting digons.  Since  $G/H$ admits a nowhere-zero $3$-flow by the hypothesis,  by Theorem~\ref{thm:Z3-connected},  $G$ admits
a nowhere-zero $3$-flow and hence, by Theorem~\ref{thm:3-R3-S1-flow}, an
$S^1$-flow.

 In the following, we assume that $W_n$ is an induced subgraph of $G$ and $n$ is odd. We also assume that \(G\) does not admit a nowhere-zero \(3\)-flow,
since otherwise the result again follows immediately from
Theorem~\ref{thm:3-R3-S1-flow}.
By hypothesis, \(G/W_n\) admits a nowhere-zero \(3\)-flow. Therefore,
Lemma~\ref{lem:wn} implies that $G-E(W_n)$ admits a nowhere-zero \(3\)-flow.

Because $G$ is $2$-connected and $W_n$ is a proper induced subgraph,
there are two distinct vertices $p,q\in V(W_n)$ that belong to the same
component of $G-E(W_n)$. Indeed, every component of $G-V(W_n)$ has at
least two neighbors in $W_n$; otherwise, its unique neighbor in $W_n$
would be a cut vertex of $G$.

By Lemma~\ref{lem:sqrt3-preflow}, there exist a vector
$\bm b\in\mathbb R^2$ with $\lVert\bm b\rVert=\sqrt{3}$ and an
$S^1$-preflow $F_1$ on $G-E(W_n)$ such that
$\partial F_1(p)=\bm b$ and $\partial F_1(q)=-\bm b$, while
$\partial F_1(v)=\bm 0$ for every
$v\in V(G)\setminus\{p,q\}$.

Applying Lemma~\ref{lem:odd-wheel-any-absorber} to $W_n$, with the same
vertices $p,q$ and the same vector $\bm b$, gives an $S^1$-preflow
$F_2$ on $W_n$ such that
$\partial F_2(p)=-\bm b$ and $\partial F_2(q)=\bm b$, while
$\partial F_2(v)=\bm 0$ for every
$v\in V(W_n)\setminus\{p,q\}$.

The edge sets of $G-E(W_n)$ and $W_n$ partition $E(G)$, so $F_1$
and $F_2$ combine to give an $S^1$-preflow $F$ on $G$. Since
$\partial F_1(p)+\partial F_2(p)=\bm0$ and
$\partial F_1(q)+\partial F_2(q)=\bm0$, while the boundary is zero at
every other vertex, $F$ is an $S^1$-flow of $G$.
\end{proof}

\section{\(S^1\)-flows in triangularly connected graphs}

In this section, we prove
Theorem~\ref{thm:s1-characterization}\ref{thm:s1-characterization:triangular}.
Fan et al.~\cite{Fan2008} characterized triangularly connected graphs
that do not admit a nowhere-zero \(3\)-flow. Their characterization is
recursive and is based on odd wheels and \(2\)-sums.

\begin{theorem}[Fan et al.~\cite{Fan2008}]
\label{thm:fan}
Let \(G\) be a triangularly connected graph with \(|V(G)|\geq 3\).
Then \(G\) does not admit a nowhere-zero \(3\)-flow if and only if either
\(G\) is an odd wheel, or
\(
G=G_1\oplus_2 W,
\)
where \(W\) is an odd wheel and \(G_1\) is a triangularly connected graph
that does not admit a nowhere-zero \(3\)-flow.
\end{theorem}

\begin{lemma}
\label{lem:odd-wheel-minus-edge}
Let \(W\) be an odd wheel. Then the following hold.

\medskip \noindent
{\makeatletter\def\@currentlabel{(i)}\label{lem:odd-wheel-minus-edge:i}\makeatother}
(i) For every edge \(e\in E(W)\), both \(W-e\) and \(W/e\) admit a nowhere-zero \(3\)-flow.

\medskip \noindent
{\makeatletter\def\@currentlabel{(ii)}\label{lem:odd-wheel-minus-edge:ii}\makeatother}
(ii) \(W\) does not admit an \(S^1\)-flow.
\end{lemma}

\begin{proof}
\ref{lem:odd-wheel-minus-edge:i} Since an odd wheel does not admit a nowhere-zero \(3\)-flow, by Tutte's flow polynomial theorem again,  \(W-e\) admits a nowhere-zero \(3\)-flow if and only if \(W/e\) does. It therefore suffices to consider \(H=W/e\).
The graph \(H\) can be reduced to a single vertex by successively contracting digons. By Theorem~\ref{thm:Z3-connected}, \(H\) admits a nowhere-zero \(\mathbb Z_3\)-flow, and hence a nowhere-zero \(3\)-flow.

\medskip
\ref{lem:odd-wheel-minus-edge:ii} If $W=K_4$, then $W$ does not admit a nowhere-zero $3$-flow. Since $K_4$ is cubic, Theorem~\ref{thm:3-R3-S1-flow} implies that $W$ does not admit an $S^1$-flow. Now assume that $|V(W)|\geq 5$. Then $W[V_3(W)]$ is a cycle and hence connected, while $W-V_3(W)$ consists of a single vertex and is therefore acyclic. Since an odd wheel does not admit a nowhere-zero $3$-flow, Theorem~\ref{thm:wang}\ref{thm:wang-degree3} implies that $W$ does not admit an $S^1$-flow.
\end{proof}

\begin{proposition}
\label{prop:2-sum-wheel}
Let $G=G_1\oplus_e W$, where $W$ is a wheel. If either $G_1$
admits an $S^1$-flow or $G_1$ is bridgeless and $G_1/e$ admits a nowhere-zero $3$-flow, then $G$ admits an $S^1$-flow.
\end{proposition}

\begin{proof}
If $W$ is an even wheel, then $W$ admits a nowhere-zero $3$-flow, and hence
so does $W/e$.  If $W$ is an odd wheel, by 
Lemma~\ref{lem:odd-wheel-minus-edge}\ref{lem:odd-wheel-minus-edge:i}, $W/e$ admits a nowhere-zero $3$-flow.
Thus, in either case, $W/e$ admits a nowhere-zero $3$-flow. The result now
follows immediately from Theorem~\ref{thm:intro-2-sum-general}.
\end{proof}

We are now ready to prove
Theorem~\ref{thm:s1-characterization}\ref{thm:s1-characterization:triangular}.

\begin{proof}[\bf Proof of Theorem~\ref{thm:s1-characterization}\ref{thm:s1-characterization:triangular}]
By Lemma~\ref{lem:odd-wheel-minus-edge}\ref{lem:odd-wheel-minus-edge:ii}, if $G$ is an odd wheel, then it does not admit an $S^1$-flow. 

Now we prove the converse by contradiction.  Let $G$ be a counterexample with $|V(G)|$  minimum. Then $G$ is 
 triangularly connected,  is not
an odd wheel, and  does not admit an \(S^1\)-flow. By Theorem~\ref{thm:3-R3-S1-flow}, $G$
does not admit a nowhere-zero \(3\)-flow.

By Theorem~\ref{thm:fan},  $G =G_1\oplus_e W$, where $W$ is an odd wheel and \(G_1\) is triangularly connected and does not admit a nowhere-zero \(3\)-flow. 
Since \(G\) does not admit an \(S^1\)-flow,
Proposition~\ref{prop:2-sum-wheel} implies that \(G_1\) does not admit
an \(S^1\)-flow. Moreover, $|V(G_1)|<|V(G)|$. The minimality of \(G\) therefore forces \(G_1\) to be an odd wheel.
Lemma~\ref{lem:odd-wheel-minus-edge} then implies that \(G_1/e\) admits
a nowhere-zero \(3\)-flow. Thus by Theorem~\ref{thm:intro-2-sum-general}\ref{thm:intro-2-sum-general:iv}, $G$ admits an $S^1$-flow, a contradiction.
This completes the proof.
\end{proof}

\section{$S^1$-flows in graphs with a spanning triangle-tree}

In this section, we prove Theorem~\ref{thm:s1-characterization}\ref{thm:s1-characterization:crystal}. We first introduce some preliminary
terminology and results on triangle-paths and crystals.

\subsection{Outerplane embedding of triangle-paths}

We first describe an outerplane embedding of a triangle-path that will be
used in the proof of the result on odd crystals.

Let $\mathcal P=\mathcal P(u,v)$ be a triangle-path with at least five vertices.
Its triangles can be ordered as
$T_1,T_2,\ldots,T_m$ such that consecutive triangles share exactly one edge, while any two nonconsecutive triangles are edge-disjoint. The edges of $\mathcal{P}$ that belong to exactly one triangle form a Hamilton cycle of $\mathcal{P}$, which we denote by $C_{\mathcal{P}}$.  In fact, $C_{\mathcal P}$ is the symmetric difference of the edge sets of
$T_1,T_2,\ldots,T_m$.

Let $u$ and $v$ be the two  $2$-vertices of $\mathcal{P}$, and let
$C_\ell$ and $C_r$ denote the two $(u,v)$-paths of $C_{\mathcal{P}}$. 
A path in $\mathcal{P}$ whose vertex set is exactly the set of
$4^+$-vertices of $\mathcal{P}$ will be called a \emph{backbone}.

The following observation gives an outerplane embedding of a triangle-path and describes the structure of its backbone.

\begin{observation}
\label{obs:geometric-triangle-path}
Let $\mathcal{P}=\mathcal{P}(u,v)$ be a triangle-path with at least five
vertices. Then $\mathcal{P}$ has an outerplane embedding (see Figure~\ref{geo}) satisfying the
following properties.
\begin{enumerate}[label=(\arabic*)]
    \item
    \label{obs:geometric-triangle-path:outer-face}
    The outer face of the embedding is bounded by $C_{\mathcal{P}}$.

    \item
    \label{obs:geometric-triangle-path:backbone}
    There is a backbone $u_1u_2\cdots u_k$ of $\mathcal{P}$ such that the
    following hold. Let $u_0:=u$ and $u_{k+1}:=v$.
    \begin{enumerate}[label=(\roman*)]
        \item
        \label{obs:geometric-triangle-path:backbone-edges}
        For each $1\le i<k$, the edge $u_iu_{i+1}$ is a chord of
        $C_{\mathcal{P}}$ whose ends lie on opposite boundary paths
        $C_\ell$ and $C_r$. Moreover, $u$ is adjacent to $u_1$ and
        $v$ is adjacent to $u_k$.

        \item
        \label{obs:geometric-triangle-path:Qi}
        For each $1\le i\le k$, let $Q_i$ be the
        $(u_{i-1},u_{i+1})$-subpath of $C_{\mathcal{P}}$ whose interior
        contains no backbone vertex. Every internal vertex of $Q_i$ has
        degree three in $\mathcal{P}$ and is adjacent to $u_i$.
    \end{enumerate}

    \item
    \label{obs:geometric-triangle-path:leaf}
    Each leaf of $\mathcal{P}$ is adjacent to one backbone vertex and one
    $3$-vertex.

    \item
    \label{obs:geometric-triangle-path:no-4}
    If $\mathcal{P}$ has no $4$-vertex, then each internal backbone vertex
    $u_i$ $(2\le i\le k-1)$ has precisely two $4^+$-neighbors, namely
    $u_{i-1}$ and $u_{i+1}$; its other $d_{\mathcal{P}}(u_i)-2$ neighbors
    have degree three. In this case, every $Q_i$ contains an internal
    $3$-vertex. Moreover, if $k\ge2$, then each of $Q_1$ and $Q_k$
    contains at least two internal $3$-vertices.
\end{enumerate}
\end{observation}

\begin{figure}[htbp]
\centering
\begin{tikzpicture}[scale=0.6, v_main/.style={circle, draw=black!80, fill=white, inner sep=0pt, minimum size=1.2mm, thick},
    v_back/.style={circle, draw=black!80, fill=black!80, inner sep=0pt, minimum size=1.5mm},
    font=\small]

    \def\R{4.5}

    \foreach \i in {0,1,2,...,21} {
        \coordinate (p\i) at ({90 - \i*(360/22)}:\R);
    }

    \draw[line width=0.8pt, black] (p0) \foreach \x in {1,...,21} { -- (p\x) } -- cycle;

    \begin{scope}[blue, line width=0.6pt]

        \draw (p1) -- (p21); \draw (p1) -- (p20); \draw (p1) -- (p19);
        \node[blue, rotate=-20] at ({90 + 1.5*(360/22)}: \R*0.7) {$\dots$};

        \draw (p18) -- (p2); \draw (p18) -- (p3); \draw (p18) -- (p4);
        \node[blue, rotate=20] at ({90 - 3*(360/22)}: \R*0.7) {$\dots$};

        \draw (p5) -- (p17); \draw (p5) -- (p16);

        \draw (p15) -- (p6); \draw (p15) -- (p7); \draw (p15) -- (p8); \draw (p15) -- (p9);
        \node[blue, rotate=-40] at ({90 - 7.5*(360/22)}: \R*0.7) {$\dots$};

        \draw (p10) -- (p14); \draw (p10) -- (p13); \draw (p10) -- (p12);
        \node[blue, rotate=40] at ({90 - 13.5*(360/22)}: \R*0.7) {$\dots$}; 
    \end{scope}

    \begin{scope}[red, line width=1.2pt]
        \draw (p1) -- (p18);
        \draw (p18) -- (p5);
 
        \node[red, font=\Large] at (0,-0.5) {$\vdots$};
        \draw (p15) -- (p10);
    \end{scope}

     \node[v_back, label=above:{$u$}] at (p0) {};
    \node[v_back, label=right:{$u_1$}] at (p1) {};
    \node[v_back, label=left:{$u_2$}] at (p18) {};
    \node[v_back, label=right:{$u_3$}] at (p5) {};
    \node[v_back, label=left:{$u_{k-1}$}] at (p15) {};
    \node[v_back, label=right:{$u_k$}] at (p10) {};
    \node[v_back, label=below:{$v$}] at (p11) {};

    \foreach \i in {2,3,4,6,7,8,9,12,13,14,16,17,19,20,21} {
        \node[v_main] at (p\i) {};
    }

    \node at ({90 - 3.5*(360/22)}: \R+0.3) {$\dots$};
    \node at ({90 - 7.5*(360/22)}: \R+0.3) {$\dots$};
    \node at ({90 - 16.5*(360/22)}: \R+0.3) {$\dots$};
    \node at ({90 - 13.5*(360/22)}: \R+0.3) {$\dots$};
    \node at ({90 - 20*(360/22)}: \R+0.3) {$\dots$};

\end{tikzpicture}
\caption{A outerplane embedding of a triangle-path}
\label{geo}
\end{figure}
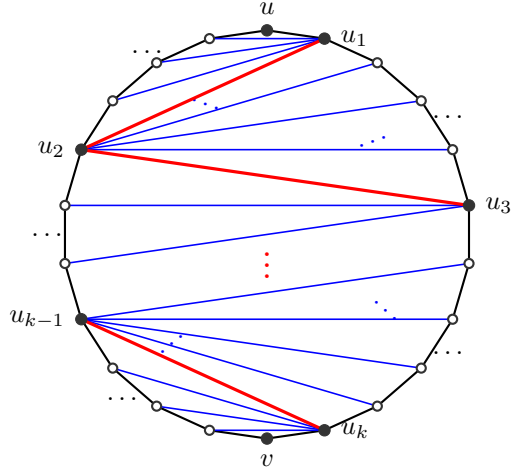

\begin{proof}
Recall that $T_1,T_2,\ldots,T_m$ are the triangles of $\mathcal P$,
ordered so that $T_i$ and $T_{i+1}$ share an edge for each $1\le i<m$,
while any two nonconsecutive triangles are edge-disjoint. Draw $T_1$
in the plane, and successively add each $T_{i+1}$ on the side of the
edge shared by $T_i$ and $T_{i+1}$ opposite to the previously drawn
part. This gives an outerplane embedding whose outer face is bounded
by $C_{\mathcal P}$. Thus \ref{obs:geometric-triangle-path:outer-face}
holds.

For $1\le j<m$, let $e_j$ be the edge shared by $T_j$ and $T_{j+1}$.
For $2\le j\le m-1$, let $w_j$ be the common end of $e_{j-1}$ and
$e_j$. Note that the vertices $w_2,\ldots,w_{m-1}$ need not be
distinct. Moreover, if $w_j=w_{j'}$ for some
$2\le j<j'\le m-1$, then
$w_j=w_{j+1}=\cdots=w_{j'}$. Thus the sequence
$w_2,\ldots,w_{m-1}$ consists of maximal consecutive subsequences of
equal vertices. Let $u_1,\ldots,u_k$ be the vertices corresponding to
these subsequences, in their order, and let $r_i$ be the length of the
subsequence corresponding to $u_i$.

A vertex that belongs to exactly $t$ consecutive triangles has degree
$t+1$ in $\mathcal P$. Hence $u_1,\ldots,u_k$ are precisely the
$4^+$-vertices of $\mathcal P$, and
\(
d_{\mathcal P}(u_i)=r_i+3.
\)
Suppose that $u_i$ and $u_{i+1}$ correspond to two consecutive
subsequences. Then, for some $j$, we have $w_j=u_i$ and
$w_{j+1}=u_{i+1}$. By the definitions of $w_j$ and $w_{j+1}$, the two
vertices $u_i$ and $u_{i+1}$ are the ends of $e_j$. Hence
$u_iu_{i+1}$ is a chord of $C_{\mathcal P}$. Moreover, the two ends of
each $e_j$ lie on opposite $(u,v)$-paths $C_\ell$ and $C_r$.
Consequently, the backbone vertices alternate between $C_\ell$ and
$C_r$. Also, $u$ is adjacent to $u_1$ and $v$ is adjacent to $u_k$.
Therefore $u_1u_2\cdots u_k$ is a backbone satisfying
\ref{obs:geometric-triangle-path:backbone}\ref{obs:geometric-triangle-path:backbone-edges}.

The $r_i$ consecutive occurrences of $u_i$ correspond to $r_i+2$
consecutive triangles containing $u_i$. The vertices introduced between
these occurrences lie in the interior of $Q_i$, are adjacent to $u_i$,
and belong to exactly two consecutive triangles. Hence they have degree
three in $\mathcal P$.
Thus
\ref{obs:geometric-triangle-path:backbone}\ref{obs:geometric-triangle-path:Qi}
holds, and therefore
\ref{obs:geometric-triangle-path:backbone} follows.

The first and last triangles in the sequence show that every leaf of
$\mathcal P$ is adjacent to one backbone vertex and one $3$-vertex.
Thus \ref{obs:geometric-triangle-path:leaf} holds.

Finally, suppose that $\mathcal P$ has no $4$-vertex. Then
$d_{\mathcal P}(u_i)\ge5$ for every $i$, and hence $r_i\ge2$ for every
$i$. It follows that each $Q_i$ contains an internal vertex, while
$Q_1$ and $Q_k$ each contain at least two internal vertices when
$k\ge2$. By
\ref{obs:geometric-triangle-path:backbone}\ref{obs:geometric-triangle-path:Qi},
every internal vertex of $Q_i$ has degree three. If
$2\le i\le k-1$, then every neighbor of $u_i$ other than
$u_{i-1}$ and $u_{i+1}$ belongs to the interior of one of
$Q_{i-1}$, $Q_i$, and $Q_{i+1}$, and hence has degree three. Therefore,
$u_{i-1}$ and $u_{i+1}$ are precisely the $4^+$-neighbors of $u_i$.
Thus \ref{obs:geometric-triangle-path:no-4} holds.
\end{proof}

\subsection{Properties of odd crystals}

We first establish some preliminary results.

\begin{observation}\label{ob:pi/3}
Let $\bm z_1, \bm z_2, \bm z_3$ be three vectors in $S^1$. If there exist three integers $\varepsilon_1, \varepsilon_2,\varepsilon_3 \in \{-1, 1\}$ such that
 $\varepsilon_1\bm z_1 + \varepsilon_2\bm z_2 + \varepsilon_3\bm z_3 = \bm 0$, then 
 the set
$\{\pm \bm z_1,  \pm \bm z_2, \pm\bm z_3\}$ is the vertex set of a regular
hexagon inscribed in the unit circle.
\end{observation}
\begin{proof}
Since  $\varepsilon_1\bm z_1 + \varepsilon_2\bm z_2 + \varepsilon_3\bm z_3 = \bm 0$ and each $\varepsilon_i\bm z_i \in S^1$, 
 the three vectors
$\varepsilon_1\bm z_1,  \varepsilon_2\bm z_2, \varepsilon_3\bm z_3$ form an equilateral triangle
centered at the origin. Hence their antipodal points together with the
three original vectors give six equally spaced points on the unit circle.
Therefore, $\{\pm \bm z_1,  \pm \bm z_2, \pm\bm z_3\} =\{\pm \varepsilon_1\bm z_1,  \pm  \varepsilon_2\bm z_2, \pm \varepsilon_3\bm z_3\}$
is the vertex set of a regular hexagon.
\end{proof}

\begin{lemma}
\label{lem:regular-hexagon-extension}
Let $k\ge1$, and let $\bm z_1,\ldots,\bm z_k,\bm a,\bm b \in S^1$ satisfy
$\bm z_1+\cdots+\bm z_k+\bm a+\bm b=\bm 0$
and
$\bm a+\bm b\ne\bm 0$.
Suppose that $\bm z_1,\ldots,\bm z_k$ are vertices of a fixed regular hexagon
inscribed in the unit circle. Then $\bm a$ and $\bm b$ are also vertices of
the same regular hexagon.
\end{lemma}

\begin{proof}
Without loss of generality, identify the plane $\R^2$ with the complex
plane $\mathbb{C}$ and assume that the regular hexagon is centered at the
origin. After a rotation, we may suppose that its vertices are
$\Omega=\{\omega^j\mid j=0,1,\ldots,5\}$, where $\omega=e^{i\pi/3}$.
Notice that $\Omega$ is precisely the set of elements of norm $1$ in the
ring of Eisenstein integers
$\mathbb{Z}[\omega]=\{x+y\omega\mid x,y\in\mathbb{Z}\}$. In particular,
$\mathbb{Z}[\omega]$ is closed under addition.

Since $\bm z_1,\ldots,\bm z_k\in\Omega$, their sum
$S=\bm z_1+\cdots+\bm z_k$ belongs to $\mathbb{Z}[\omega]$. By the given
condition, $\bm a+\bm b=-S$, and hence
$\bm v:=\bm a+\bm b\in\mathbb{Z}[\omega]$. Since $\bm a$ and $\bm b$ are
unit vectors, the triangle inequality gives $\lVert\bm v\rVert\le2$.
As $\bm v\neq\bm0$, we have
$\lVert\bm v\rVert^2\in\{1,2,3,4\}$. Moreover, for every $x+y\omega\in\mathbb{Z}[\omega]$,
$\lVert x+y\omega\rVert^2=x^2+xy+y^2\equiv(x-y)^2\pmod3$.
Since a square is congruent to either $0$ or $1$ modulo $3$, we have
$\lVert\bm v\rVert^2\not\equiv2\pmod3$. Hence
$\lVert\bm v\rVert^2\neq2$, and we are left with three cases.

\medskip
\noindent\textbf{Case 1: $\lVert\bm v\rVert^2=1$.}
In this case, $\bm v$ is a unit element of $\mathbb{Z}[\omega]$, and hence
$\bm v\in\Omega$. Since $\bm a$ and $\bm b$ are unit vectors whose sum is
the unit vector $\bm v$, the angle between $\bm a$ and $\bm b$ is
$2\pi/3$. Equivalently, $\bm a$ and $\bm b$ are obtained from $\bm v$
by rotations through angles $\pm\pi/3$. Multiplication by
$e^{\pm i\pi/3}=\omega^{\pm1}$ permutes $\Omega$, and therefore
$\bm a,\bm b\in\Omega$.

\medskip
\noindent\textbf{Case 2: $\lVert\bm v\rVert^2=3$.}
Here, $\lVert\bm v\rVert=\sqrt{3}$. The only way two unit vectors can sum to a vector
of length $\sqrt{3}$ is for the angle between them to be $\pi/3$ (since
$1^2+1^2+2\cos(\pi/3)=3$). Thus, $\bm a$ and $\bm b$ are separated by
$\pi/3$, which is exactly the angle between adjacent vertices of the regular
hexagon. Solving $x^2+xy+y^2=3$ for integers $x,y$, we find exactly six
possible values for $\bm v=x+y\omega$, namely
$\pm(1+\omega)$, $\pm(-1+2\omega)$, and $\pm(-2+\omega)$. Using the relation
$\omega^2=\omega-1$, three of these can be explicitly written as sums of
adjacent vertices in $\Omega$: $1+\omega=\omega^0+\omega^1$,
$-1+2\omega=\omega^1+\omega^2$, and
$-2+\omega=\omega^2+\omega^3$. The remaining three values are their
negations, corresponding to the sums of the antipodes of those adjacent
pairs (e.g., $-(1+\omega)=\omega^3+\omega^4$). Because the geometric
decomposition of a vector of length $\sqrt{3}$ into two unit vectors is
unique, $\bm a$ and $\bm b$ must be exactly the two adjacent vertices
involved in the sum. Hence, $\bm a,\bm b\in\Omega$.

\medskip
\noindent\textbf{Case 3: $\lVert\bm v\rVert^2=4$.}
In this case, $\lVert\bm v\rVert=2$. Since $\bm a$ and $\bm b$ are unit vectors, they
can only sum to a vector of length $2$ if $\bm a=\bm b$, yielding
$\bm v=2\bm a$. This implies $\bm a=\bm v/2$. Let $\bm v=x+y\omega$.
Since $x^2+xy+y^2=4$, reducing modulo $2$ shows that $x$ and $y$ must both
be even integers. Therefore,
$\bm v/2=(x/2)+(y/2)\omega\in\mathbb{Z}[\omega]$. Since
$\lVert\bm v/2\rVert^2=1$, we have $\bm a=\bm v/2\in\Omega$. Consequently,
$\bm b=\bm a\in\Omega$.

In all three cases, $\bm a,\bm b\in\Omega$. This completes the proof.
\end{proof}

\begin{lemma}\label{lem:crystal-contraction}
Let $\mathcal C=\mathcal P(u,v)+uv$ be an odd crystal. For every edge
$e\in E(\mathcal C)$, the graph $\mathcal C/e$ admits a nowhere-zero
$3$-flow.
\end{lemma}

\begin{proof}
If $e\neq uv$, then the edge $e$ belongs to a triangle of the triangle-path $\mathcal{P}(u,v)$. Contracting $e$  results in  a digon. By successively  contracting  digons, $\mathcal{C}/e$ can be reduced to a single vertex. Thus $\mathcal{C}/e$ admits a nowhere-zero $3$-flow by Theorem~\ref{thm:Z3-connected}.

If $e=uv$, then $\mathcal{C}/e$  can be obtained from $\mathcal{P}(u,v)$ by
identifying $u$ and $v$. Let $G_1$ be the graph obtained from
$\mathcal{P}(u,v)$ by suppressing its $2$-vertices. Then $G_1$ can
be reduced to a single vertex by successively contracting digons. By  Theorem~\ref{thm:Z3-connected},  $G_1$ admits a nowhere-zero
$3$-flow and so do $\mathcal{P}(u,v)$ and $\mathcal{C}/e$. 
\end{proof}

\begin{corollary}\label{cor:crystal-wheel-sum}
Suppose that $\mathcal{C}=\mathcal{P}(u,v)+uv$ is an odd crystal. For every
$e\in E(\mathcal C)$ and every wheel $W$ with a distinguished edge also
denoted by $e$, the $2$-sum $G:=\mathcal{C}\oplus_e W$ admits an
$S^1$-flow.
\end{corollary}

\begin{proof}
By Lemma~\ref{lem:crystal-contraction}, $\mathcal{C}/e$ admits a
nowhere-zero $3$-flow. Since $\mathcal C$ is bridgeless,
Proposition~\ref{prop:2-sum-wheel} implies
that $G=\mathcal{C}\oplus_e W$ admits an $S^1$-flow.
\end{proof}

We now establish the following  property of odd crystals.

\begin{lemma}\label{bullgrowth}
Suppose that $\mathcal{C}=\mathcal{P}(s,t)+st$ is an odd crystal. Then one
of the following holds:
\begin{enumerate}[label=(\arabic*)]
    \item $\mathcal{B}\uplus\mathcal{C}$ contains no spanning triangle-tree;
    \item $\mathcal{B}\uplus\mathcal{C}$ is also an odd crystal;
    \item $\mathcal{B}\uplus\mathcal{C}$ admits an $S^1$-flow.
\end{enumerate}
\end{lemma}

\begin{proof}
Let $G=\mathcal{B}\uplus\mathcal{C}$, and recall the notation from
Definition~\ref{def:bull-reduction}. In particular, $u$ and $v$ are the two
new adjacent $3$-vertices introduced by the bull-growth operation, $w$ is
their common neighbor, and $a$ and $b$ are the neighbors of $u$ and $v$,
respectively, other than $w$ and each other. If $a\ne b$, then the
bull-growth operation replaces the edge $ab$ of $\mathcal{C}$ by the bull
configuration; if $a=b$, then the edge $ab$ does not exist in $\mathcal{C}$.

Note that, by this construction, every vertex of $G$ has odd degree. We note that a triangle-tree $T_0$ has
$2|V(T_0)|-3$ edges and  at least two
$2$-vertices.

Suppose that (1) fails, and let $T$ be a spanning triangle-tree of $G$.  Put $n=|V(\mathcal C)|$, so
$|E(\mathcal C)|=2n-2$.

First suppose that $a\ne b$. Then $|V(G)|=n+2$ and
$|E(G)|=2n+2=2|V(G)|-2$, whereas $|E(T)|=2|V(G)|-3$. Thus
$T=G-xy$ for some edge $xy$. Every vertex other than $x$ and $y$ has odd
degree in $T$, so $x$ and $y$ are the only two $2$-vertices of $T$, implying that $T$ is a triangle-path. Thus
$G=T+xy$ is an odd crystal, and (2) holds.

Now suppose that $a=b$. If $aw\in E(\mathcal C)$, then
$G=\mathcal C\oplus_{aw}K_4$ and thus admits an $S^1$-flow  by Corollary~\ref{cor:crystal-wheel-sum}. Hence, (3) holds.

It remains to consider the case $a=b$ and $aw\notin E(\mathcal C)$.
The only triangles of $G$ containing $u$ or $v$ are $uva$ and $uvw$. By the definition
of a triangle-tree, $T$ can be obtained from a triangle $T_0=K_3$ by
successively adding a new vertex adjacent to the two ends of an
existing edge.

We first show that $u,v\in V(T_0)$. If both $u$ and $v$ are not in
$T_0$, then, without loss of generality, suppose that $u$ is added
before $v$. Let $T'$ be the triangle-tree just before $u$ is added.
Then $u$ is adjacent to the two ends $x,y$ of some edge $xy\in E(T')$,
and hence $uxy$ is a triangle of $G$. Since $v$ has not yet been added,
$v\notin\{x,y\}$, contradicting the fact that every triangle of $G$
containing $u$ also contains $v$. On the other hand, if $u\in V(T_0)$,
then the triangle $T_0$ contains $u$ and therefore also $v$. Hence
$u,v\in V(T_0)$.

Since $u$ and $v$ are adjacent, we may assume that
$T_0=uvw$; in particular, $a$ is the only vertex outside $T_0$ that is
adjacent to two vertices of $T_0$, namely $u$ and $v$. Thus $a$ must be
the next vertex added to $T_0$. Let $T_1$ be the resulting
triangle-tree. After adding $a$, no vertex of $\mathcal C\setminus
\{a,w\}$ is adjacent to both ends of an edge of $T_1$. Hence no further
vertex can be added to $T_1$. This contradicts the fact that $T$ is
spanning, since $\mathcal C$ has at least four vertices.
\end{proof}

\subsection{Proof of Theorem~\ref{thm:s1-characterization}
\ref{thm:s1-characterization:crystal}}

Li et al.~\cite{LLW20} completely characterized graphs with a spanning
triangle-tree that admit a nowhere-zero $3$-flow. In particular, they showed
that no odd crystal admits a nowhere-zero $3$-flow.

\begin{theorem}[\cite{LLW20}]\label{llw}
Let $G$ be a graph containing a spanning triangle-tree.
Then $G$ does not admit a nowhere-zero $3$-flow if and only if either
$G=K_4$, or
$G=\mathcal{B}\uplus G_1$, where $G_1$ also contains a spanning
triangle-tree and does not admit a nowhere-zero $3$-flow. In particular,  no odd crystal admits a nowhere-zero $3$-flow.
\end{theorem}

Combining Theorem~\ref{llw} with the structural results established in the
previous sections, we now prove Theorem~\ref{thm:s1-characterization}
\ref{thm:s1-characterization:crystal}.

\begin{proof}[{\bf Proof of Theorem~\ref{thm:s1-characterization}
\ref{thm:s1-characterization:crystal}}]
We first prove that no odd crystal admits an $S^1$-flow.

Suppose, to the contrary, that 
$G=\mathcal{P}(s,t)+st$  is a counterexample with $|E(G)|$ minimum. Let
$(D,f)$ be an $S^1$-flow of $G$. Note that $K_4$ is an odd wheel, so,  by Lemma~\ref{lem:odd-wheel-minus-edge}\ref{lem:odd-wheel-minus-edge:ii}, it does not admit an $S^1$-flow.  Thus
$G\neq K_4$, and consequently $|V(G)|\geq5$.

Let $u_1\cdots u_k$ be the backbone in the outerplane embedding of
$\mathcal P=\mathcal P(s,t)$, and put $u_0=s$ and $u_{k+1}=t$.
Since $G$ is an odd crystal, the two leaves of $\mathcal P$ have degree
two and every other vertex has odd degree. In particular, $\mathcal P$
has no $4$-vertex. We may therefore use all four parts of
Observation~\ref{obs:geometric-triangle-path}; retain its notation $Q_i$. Define
$\Omega=\{\pm f(e):e\in\delta_G(s)\}$. Since $s$ is a $3$-vertex,
Observation~\ref{ob:pi/3} implies that $\Omega$ is the vertex set of a
regular hexagon inscribed in the unit circle. Call an edge of $G$ an
\emph{$\Omega$-edge} if its flow value belongs to $\Omega$. Let
$G_\Omega$ be the subgraph of $G$ induced by all $\Omega$-edges, and let
$G_\Omega(s)$ be the component of $G_\Omega$ containing $s$.

We first show that $G$ contains an edge that is not an $\Omega$-edge.
Suppose, otherwise, that all flow values belong to $\Omega$. Choose a set
$R$ of three vertices of $\Omega$ whose pairwise angular distances are
$2\pi/3$. Then $\Omega=R\cup(-R)$.  For every edge whose value lies in $-R$, reverse its orientation and replace its value by its negative. All flow values
now lie in $R$. Multiplying all
values by a suitable orthogonal matrix then yields an $R_3$-flow,
where $R_3$ is the set of cube roots of unity.  By Theorem~\ref{thm:3-R3-S1-flow}, $G$ admits a nowhere-zero
$3$-flow, contradicting Theorem~\ref{llw}. Therefore, $G$
contains an edge that is not an $\Omega$-edge.

The following fact is an immediate corollary of Lemma~\ref{lem:regular-hexagon-extension}. It will be  repeatedly applied in the remainder of the proof.

\medskip \noindent
{\bf Fact}  If a $3$-vertex is incident with an
$\Omega$-edge, then all three incident edges are $\Omega$-edges.

\medskip

Since $st$ is an $\Omega$-edge, the vertices $s$ and $t$ lie in the same
component $G_\Omega(s)$. Every vertex outside the backbone is a $3$-vertex,
and thus some backbone vertex is incident with a non-$\Omega$-edge. Let $q$
be the smallest index for which not every edge incident with $u_q$ is an
$\Omega$-edge.

We claim that $q\leq k-1$. Suppose, to the contrary, that $q=k$. Since
$t$ is a $3$-vertex incident with the $\Omega$-edge $st$, all edges
incident with $t$ are $\Omega$-edges. By the choice of $q$, every edge
incident with $u_{k-1}$ is also an $\Omega$-edge. Every edge incident
with $u_k$ is either $u_{k-1}u_k$ or is incident with a $3$-vertex
adjacent to $u_k$. Each such $3$-vertex is either adjacent to
$u_{k-1}$ or lies on $Q_k$, whose one endvertex is $t$. Applying
Lemma~\ref{lem:regular-hexagon-extension} successively to these
$3$-vertices, starting from $t$, we conclude that every edge incident
with $u_k$ is an $\Omega$-edge, contradicting the definition of $q$.
Hence $1\leq q\leq k-1$.

Let $x$ be the last internal vertex of $Q_q$ when traversed from
$u_{q-1}$ to $u_{q+1}$, and let $y$ be the first internal vertex of
$Q_{q+1}$ when traversed from $u_q$ to $u_{q+2}$. By
Observation~\ref{obs:geometric-triangle-path}\ref{obs:geometric-triangle-path:no-4}, both vertices
exist and have degree three. In particular, $x\ne s$ and
$y\ne u_{q+1}$.

All edges incident with $u_{q-1}$ are $\Omega$-edges. Applying
Lemma~\ref{lem:regular-hexagon-extension} successively to the
$3$-vertices along $Q_q$ shows that
every edge incident with $u_q$, except possibly $e_1=u_qu_{q+1}$ and
$e_2=u_qy$, is an $\Omega$-edge. Since $e_1$ and $e_2$ are distinct and,
by the definition of $q$, not all edges incident with $u_q$ are
$\Omega$-edges, at least one of $e_1$ and $e_2$ is not an $\Omega$-edge.

Without loss of generality, we assume that every edge incident with $u_q$ is directed
away from $u_q$. Then 
$\bm z_1+\cdots+\bm z_r+\bm a+\bm b=\bm0$, where
$\bm z_1,\ldots,\bm z_r\in\Omega$, $\bm a=f(e_1)$, and
$\bm b=f(e_2)$.

We claim that $\bm a+\bm b\neq\bm0$. Suppose otherwise. Split $u_q$ into
two vertices, one incident with $e_1$ and $e_2$, and the other incident with
all edges at $u_q$ other than $e_1$ and $e_2$. Since $\bm a+\bm b=\bm0$, the resulting graph $H$ inherits an $S^1$-flow
$(D',f')$ from $G$.

Notice that $x\neq s$. By the outerplane embedding of
$\mathcal{P}(s,t)$, $\{xu_{q+1},u_qu_{q+1},u_qy\}$ is a $3$-edge-cut of
$\mathcal{P}(s,t)$, and hence $xu_{q+1}$ is a cut edge of $H-st$.
Consequently, $\{st,xu_{q+1}\}$ is a nontrivial $2$-edge-cut of $H$. Let
$H_s$ and $H_t$ be the components of
$H-\{st,xu_{q+1}\}$ containing $s$ and $t$, respectively, and define
$G'=H_s+sx$. Then $G'$ is a crystal. Moreover, the splitting decreases the
degree of $u_q$ by $2$ and leaves the parity of every other vertex
unchanged. Since $G$ is an odd crystal, $G'$ is also an odd crystal.
Moreover, $|E(G')|<|E(G)|$.

The $S^1$-flow $(D',f')$ of $H$ induces an $S^1$-flow
$(D'',f'')$ of $G'$ as follows. For every edge other than the new edge
$sx$, let its orientation and flow value in $(D'',f'')$ be the same as
in $(D',f')$. If $st$ is oriented from $s$ to $t$ in $D'$, orient $sx$
from $s$ to $x$ and set $f''(sx)=f'(st)$. If $st$ is directed from $t$
to $s$ in $D'$, orient $sx$ from $x$ to $s$ and set
$f''(sx)=f'(st)$. Then $(D'',f'')$ is an $S^1$-flow of $G'$, contradicting the
minimality of $G$. 

Therefore,
$\bm a+\bm b\neq\bm0$. Applying
Lemma~\ref{lem:regular-hexagon-extension} to
$\bm z_1+\cdots+\bm z_r+\bm a+\bm b=\bm0$ now gives
$\bm a,\bm b\in\Omega$. Thus both $e_1$ and $e_2$ are $\Omega$-edges,
contradicting the choice of $u_q$. We conclude that no odd crystal admits
an $S^1$-flow.

\medskip
It remains to prove the converse. Suppose, to the contrary, that the
statement is false, and choose a counterexample $G$ with $|E(G)|$ minimum. Thus $G$ contains a spanning triangle-tree, is not an odd crystal,
and admits no $S^1$-flow. By Theorem~\ref{thm:3-R3-S1-flow}, $G$ also admits
no nowhere-zero $3$-flow.

Since $K_4$ is an odd crystal, $G\neq K_4$, and hence $|V(G)|\geq5$.
By Theorem~\ref{llw}, we have
$G=\mathcal{B}\uplus G_1$, where $G_1$ contains a spanning triangle-tree and
admits no nowhere-zero $3$-flow. By Theorem~\ref{thm:intro-bull-growth},
$G_1$ admits no $S^1$-flow. Since $|E(G_1)|<|E(G)|$, the minimality of $G$
implies that $G_1$ is an odd crystal.
Applying Lemma~\ref{bullgrowth} to the bull-growth
$G=\mathcal{B}\uplus G_1$, we obtain that at least one of the following
holds:
\begin{itemize}
\item $G$ contains no spanning triangle-tree;
\item $G$ is an odd crystal;
\item $G$ admits an $S^1$-flow.
\end{itemize}
Each of these contradicts the choice of $G$. This contradiction completes
the proof.
\end{proof}

\section{Concluding remarks}
\label{sec:concludingremarks}

\subsection{Incomparability of the two graph classes}
\label{sec:Triangularly connected graphs and spanning triangle-trees}

The class of triangularly connected graphs and the class of graphs containing
a spanning triangle-tree are incomparable. We exhibit infinite families
showing that each class contains graphs outside the other.

First clearly, every crystal $\mathcal{C} = \mathcal{P}(u,v) + uv$with at least $5$ vertices, which contains a spanning triangle-tree,   is not triangularly connected since the edge $uv$ is not contained in a triangle. 

Conversely, let $W_n$ be the wheel with hub $h$ and rim
$C_n=r_0r_1\cdots r_{n-1}r_0$, where $n\ge4$. Let $\widehat W_n$ be
obtained from $W_n$ by adding $2$-vertices
$x_0,x_1,\ldots,x_{n-1}$, where $x_i$ is adjacent to $r_i$ and
$r_{i+1}$, with subscripts taken modulo $n$. See
\Cref{fig:incomparable-b} for $\widehat W_5$.

Clearly, $\widehat W_n$ is triangularly connected. We claim that
$\widehat W_n$ contains no spanning triangle-tree. Suppose otherwise that
$T$ is a spanning triangle-tree of $\widehat W_n$. Since each $x_i$ belongs
to exactly one triangle, namely $x_ir_ir_{i+1}$, we must have
$r_ir_{i+1}\in E(T)$ for every $i$. Thus the rim
$r_0r_1\cdots r_{n-1}r_0$ is an induced cycle in $T$. Since $n\ge4$, this
contradicts the chordality of $T$.

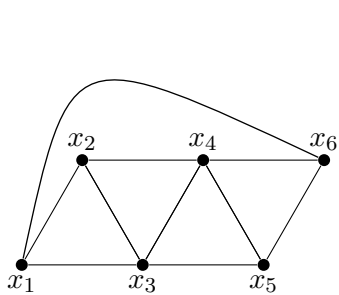
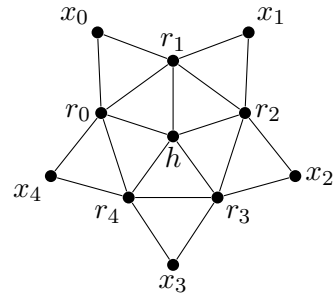
\begin{figure}[htbp]
\centering

\begin{subfigure}[b]{0.48\textwidth}
\centering
\begin{tikzpicture}[
    scale=0.8,
    every node/.style={inner sep=1.2pt},
    vtx/.style={circle, fill=black, minimum size=4.5pt, inner sep=0pt}
]

\node[vtx, label=below:$x_1$] at (0,0)            (a) {};
\node[vtx, label=below:$x_3$] at (2,0)            (b) {};
\node[vtx, label=above:$x_2$] at (1,{sqrt(3)})    (c) {};
\node[vtx, label=above:$x_4$] at (3,{sqrt(3)})    (d) {};
\node[vtx, label=below:$x_5$] at (4,0)            (e) {};
\node[vtx, label=above:$x_6$] at (5,{sqrt(3)})    (f) {};

\draw (a)--(b)--(c)--(a);
\draw (b)--(c)--(d)--(b);
\draw (b)--(d)--(e)--(b);
\draw (d)--(e)--(f)--(d);

\draw[line width=0.5pt]
        (a) .. controls +(0.8,3.7) .. (f);

\end{tikzpicture}
\caption{A crystal with a spanning triangle-tree that is not triangularly
connected.}
\label{fig:incomparable-a}
\end{subfigure}
\hfill
\begin{subfigure}[b]{0.48\textwidth}
\centering
\begin{tikzpicture}[
    scale=0.5,
    every node/.style={inner sep=1.2pt},
    vtx/.style={circle, fill=black, minimum size=4.5pt, inner sep=0pt}
]

\node[vtx, label=above:$r_1$] at (90:2)   (r1) {};
\node[vtx, label=right:$r_2$] at (18:2)   (r2) {};
\node[vtx, label=below right:$r_3$] at (-54:2)  (r3) {};
\node[vtx, label=below left:$r_4$] at (-126:2) (r4) {};
\node[vtx, label=left:$r_0$] at (162:2)  (r0) {};
\node[vtx, label=below:$h$] at (0,0)     (h)  {};

\node[vtx, label=above right:$x_1$] at (54:3.4)   (x1) {};
\node[vtx, label=right:$x_2$]       at (-18:3.4)  (x2) {};
\node[vtx, label=below:$x_3$]       at (-90:3.4)  (x3) {};
\node[vtx, label=below left:$x_4$]  at (-162:3.4) (x4) {};
\node[vtx, label=above left:$x_0$]  at (126:3.4)  (x0) {};

\draw (r1)--(r2)--(r3)--(r4)--(r0)--(r1);
\draw (h)--(r1) (h)--(r2) (h)--(r3) (h)--(r4) (h)--(r0);

\draw (x1)--(r1) (x1)--(r2);
\draw (x2)--(r2) (x2)--(r3);
\draw (x3)--(r3) (x3)--(r4);
\draw (x4)--(r4) (x4)--(r0);
\draw (x0)--(r0) (x0)--(r1);

\end{tikzpicture}
\caption{A triangularly connected graph with no spanning triangle-tree.}
\label{fig:incomparable-b}
\end{subfigure}

\caption{Examples showing the incomparability of the two graph classes.}
\label{fig:incomparable}
\end{figure}

\subsection{A conjecture relating $S^1$-flows and $4$-flows}

All currently known examples of graphs that admit an $S^1$-flow but no
nowhere-zero $3$-flow also admit a nowhere-zero $4$-flow. These observations
motivate the following conjecture.

\begin{conjecture}\label{conj:s1-four}
If a graph $G$ admits an $S^1$-flow, then $G$ admits a nowhere-zero $4$-flow.
\end{conjecture}

Combining Theorem~\ref{thm:3-R3-S1-flow}, Lemma~\ref{lem:small-cut-s1-flow}, 
Theorem~\ref{thm:intro-2-sum-general}, the Four Color Theorem~\cite{Robertson1997}, and Jaeger's $4$-flow theorem~\cite{Jaeger1979} yields
restrictions on any minimal counterexample. If such a counterexample exists,
it must be nonplanar, noncubic, and essentially $4$-edge-connected. Moreover,
it must contain both a $3$-vertex and a $5^+$-vertex, but no $2$- or
$4$-vertices. Indeed, $2$-vertices can be suppressed directly, while each
$4$-vertex can first be split into two $2$-vertices, as demonstrated in the
proof of Theorem~\ref{thm:degree4}, after which the resulting $2$-vertices
can be suppressed.

\section*{Acknowledgments}
 Jiaao Li is partially supported by National Key Research and Development Program of China (No. 2022YFA1006400), National Natural Science Foundation of China (Nos. 12571371, 12222108), Natural Science Foundation of Tianjin (No. 24JCJQJC00130), and the Fundamental Research Funds for the Central Universities, Nankai University. Rong Luo is partially supported by a grant from  Simons Foundation (No. 839830).


\begin{thebibliography}{99}
\setlength\itemsep{-0.1cm}

\bibitem{DeVosXY2006}
M.~DeVos, R.~Xu, and G.~Yu,
\emph{Nowhere-zero $\mathbb{Z}_3$-flows through $\mathbb{Z}_3$-connectivity},
Discrete Math. \textbf{306} (2006), 26--30.

\bibitem{Fan2008}
G.-H.~Fan, H.-J.~Lai, R.~Xu, C.-Q.~Zhang, and C.-X.~Zhou,
\emph{Nowhere-zero $3$-flows in triangularly connected graphs},
J. Combin. Theory Ser. B \textbf{98} (2008), 1325--1336.

\bibitem{GMRR2025}
L.~Gáborik, S.~Kurz, G.~Mazzuoccolo, J.~Rajník, and F.~Rieg,
\emph{Manhattan and Chebyshev flows},
arXiv preprint arXiv:2510.22234, 2025.

\bibitem{Goddyn}
L.~A.~Goddyn, M.~Tarsi, and C.-Q.~Zhang,
\emph{On $(k,d)$-colorings and fractional nowhere-zero flows},
J. Graph Theory \textbf{28} (1998), 155--161.

\bibitem{HMM2026}
H.~Houdrouge, B.~Miraftab, and P.~Morin,
\emph{$2$-dimensional unit vector flows},
arXiv preprint arXiv:2602.21526, 2026.

\bibitem{Imrich2010}
W.~Imrich, I.~Peterin, S.~\v{S}pacapan, and C.-Q.~Zhang,
\emph{NZ-flows in strong products of graphs},
J. Graph Theory \textbf{64} (2010), 267--276.

\bibitem{Jaeger1979}
F.~Jaeger,
\emph{Flows and generalized coloring theorems in graphs},
J. Combin. Theory Ser. B \textbf{26} (1979), 205--216.

\bibitem{Jain2007}
K.~Jain,
\emph{Unit vector flows},
Open Problem Garden (2007).
Available at \url{http://www.openproblemgarden.org/op/unit_vector_flows}.

\bibitem{Lai2000}
H.-J.~Lai,
\emph{Group connectivity of $3$-edge-connected chordal graphs},
Graphs Combin. \textbf{16} (2000), 165--176.

\bibitem{Lai2003}
H.-J.~Lai,
\emph{Nowhere-zero $3$-flows in locally connected graphs},
J. Graph Theory \textbf{42} (2003), 211--219.

\bibitem{LLLS2026}
C.~Li, J.~Li, R.~Luo, and B.~Su,
\emph{High-dimensional $p$-normed flows},
arXiv preprint arXiv:2601.12036, 2026.

\bibitem{LLW20}
J.~Li, X.~Li, and M.~Wang,
\emph{Spanning triangle-trees and flows of graphs},
Graphs Combin. \textbf{36} (2020), 1797--1814.

\bibitem{mattiolo2023d}
D.~Mattiolo, G.~Mazzuoccolo, J.~Rajník, and G.~Tabarelli,
\emph{On $d$-dimensional nowhere-zero $r$-flows on a graph},
European J. Math. \textbf{9} (2023), 101.

\bibitem{mattiolo2024lower}
D.~Mattiolo, G.~Mazzuoccolo, J.~Rajník, and G.~Tabarelli,
\emph{A lower bound for the complex flow number of a graph: A geometric approach},
J. Graph Theory \textbf{106} (2024), 239--256.

\bibitem{mattiolo2025geometric}
D.~Mattiolo, G.~Mazzuoccolo, J.~Rajník, and G.~Tabarelli,
\emph{Geometric description of $d$-dimensional flows of a graph},
Australas. J. Combin. \textbf{94} (2026), 376--384.

\bibitem{Robertson1997}
N.~Robertson, D.~Sanders, P.~Seymour, and R.~Thomas,
\emph{The four-colour theorem},
J. Combin. Theory Ser. B \textbf{70} (1997), 2--44.

\bibitem{thomassen2014group}
C.~Thomassen,
\emph{Group flow, complex flow, unit vector flow, and the $(2+\epsilon)$-flow conjecture},
J. Combin. Theory Ser. B \textbf{108} (2014), 81--91.

\bibitem{Tutte1949}
W.~T.~Tutte,
\emph{On the imbedding of linear graphs in surfaces},
Proc. London Math. Soc. \textbf{51} (1949), 474--483.

\bibitem{Tutte54}
W.~T.~Tutte,
\emph{A contribution to the theory of chromatic polynomials},
Canad. J. Math. \textbf{6} (1954), 80--91.

\bibitem{Ulyanov2026}
N.~Ulyanov,
\emph{Graph Puzzles II.1: Counterexamples to Jain's Second Unit Vector Flows Conjecture},
arXiv preprint arXiv:2603.23328, 2026.

\bibitem{wang2015vector}
Y.~Wang, J.~Cheng, R.~Luo, and C.-Q.~Zhang,
\emph{Vector flows and integer flows},
SIAM J. Discrete Math. \textbf{29} (2015), 2166--2178.

\bibitem{Xu2004}
R.~Xu,
\emph{On Flows of Graphs},
Ph.D. dissertation, West Virginia University, 2004.

\bibitem{zhang1997integer}
C.-Q.~Zhang,
\emph{Integer Flows and Cycle Covers of Graphs},
CRC Press, Boca Raton, 1997.

\end{thebibliography}
\end{document}